\def\arxivorIP{0}

\def\myabstract{
Regularization is a core component of modern inverse problems, as it helps establish the well-posedness of the solution of interest. Popular 
regularization approaches include variational regularization and iterative regularization. The former involves solving a variational optimization 
problem, which consists of a data-fidelity term and a regularization term, balanced by an appropriate weighting parameter. The latter mitigates 
overfitting to noise by selecting a suitable stopping time during the iterative process.
A key topic in the study of regularization is the relationship between the regularized solution and the original ground truth. When the ground truth 
possesses a low-complexity structure referred to as the ``model''  it can be shown that under appropriate regularization promoting the same structure, 
the solution to the regularized problem is robust to small perturbations. This property is called ``model consistency''.
For variational regularization, model consistency in linear inverse problems has been studied in~\cite{vaiter2017model}. However, for iterative regularization, 
model consistency remains an open question. In this paper, building on recent developments in partial smoothness~\cite{vaiter2017model}, 
we show that when the noise level is sufficiently small and an appropriate stopping criterion is used, iterative
regularization is model consistent as well. Moreover, we show that the considered algorithm exhibits local linear behavior of the regularization.
We provide numerical simulations to support our theoretical findings.}

\def\mykeywords{
linear inverse problems, iterative regularization, dual gradient descent, partial smoothness, model consistency.
}

\newcommand{\pa}[1]{(#1)}
\newcommand{\Pa}[1]{\big({#1}\big)}
\newcommand{\bPa}[1]{\mathlarger{\big(}{#1}\mathlarger{\big)}}

\newcommand{\Ppa}[1]{\left({#1}\right)}

\newcommand{\Ba}[1]{\big\{#1\big\}}
\newcommand{\bBa}[1]{\Big\{#1\Big\}}

\newcommand{\norm}[1]{{|\kern-1.125pt|} #1 {|\kern-1.125pt|}}

\newcommand{\msum}{\mathbin{\scalebox{1.25}{\ensuremath{\sum}}}}

\newcommand{\qfrac}[2]{%
    {\mathsmaller{\frac{\raisebox{0.15em}{\small $#1$}}{\raisebox{-0.25em}{\small $#2$}}}}
}
\newcommand{\sfrac}[2]{%
    {\mathsmaller{\frac{\raisebox{0.05em}{\footnotesize $#1$}}{\raisebox{-0.15em}{\footnotesize $#2$}}}}
}

\newcommand{\qandq}{ \quad \textrm{and} \quad }

\newcommand{\bbR}{\mathbb{R}}
\newcommand{\bbN}{\mathbb{N}}

\newcommand{\calT}{\mathcal{T}}
\newcommand{\calM}{\mathcal{M}}

\newcommand{\Msol}{{\mathcal{M}_{\wsol}}}
\newcommand{\calU}{\mathcal{U}}
\newcommand{\calY}{\mathbb{R}^{n}}
\newcommand{\calX}{{X}}
\newcommand{\calW}{\mathbb{R}^{p}}
\newcommand{\calI}{\mathcal{I}}
\newcommand{\calB}{\mathcal{B}}
\newcommand{\calF}{\mathcal{F}}

\newcommand{\find}{\mathrm{find}\quad}
\newcommand{\aff}{\mathrm{aff}}
\newcommand{\ri}{\mathrm{ri}}

\newcommand{\Id}{\mathrm{Id}}
\newcommand{\prox}{\mathrm{prox}}
\newcommand{\qstq}{\quad\mathrm{subject~to}\quad}
\newcommand{\lsc}{\Gamma_0(\calW)}
\newcommand{\bdy}{\mathrm{bdy}}
\newcommand{\supp}{\mathrm{supp}}
\newcommand{\rank}{\mathrm{rank}}
\newcommand{\dist}{\mathrm{dist}}
\newcommand{\Ker}{\mathrm{ker}}

\newcommand{\DDIF}{D_{\mathrm{_{DIF}}}}
\newcommand{\SC}{\mathrm{SC}}
\newcommand{\INJ}{\mathrm{INJ}}

\newcommand{\yzero}{{y^{\dagger}}}
\newcommand{\yd}{{y_{_{\delta}}}}
\newcommand{\wsol}{w^{\dagger}}
\newcommand{\vsol}{v^{\dagger}}
\newcommand{\zsol}{z^{\dagger}}

\newcommand{\vdsol}{v_{\delta}^{\dagger}}
\newcommand{\pldw}{p_{\lambda,\delta}(w)}
\newcommand{\Pld}{P_{\lambda,\delta}}
\newcommand{\Pzero}{P_{0}}

\newcommand{\Pd}{P_{\delta}}
\newcommand{\Ra}{R_{\alpha}}
 
\newcommand{\wk}{w^{(k)}}

\newcommand{\wkmd}{w^{(k-1)}_{\delta}}
\newcommand{\wkd}{w^{(k)}_{\delta}}
\newcommand{\wkpd}{w^{(k+1)}_{\delta}}
\newcommand{\wkmzero}{w^{(k-1)}_{0}}
\newcommand{\wkzero}{w^{(k)}_{0}}

\newcommand{\wKmd}{w^{(\Km)}_{\delta}}

\newcommand{\vkmd}{v^{(k-1)}_{\delta}}
\newcommand{\vkd}{v^{(k)}_{\delta}}
\newcommand{\vkpd}{v^{(k+1)}_{\delta}}
\newcommand{\vkmzero}{v^{(k-1)}_{0}}
\newcommand{\vkzero}{v^{(k)}_{0}}

\newcommand{\vzerod}{v^{(0)}_{\delta}}

\newcommand{\zkd}{z^{(k)}_{\delta}}

\newcommand{\zkzero}{z^{(k)}_{0}}

\newcommand{\uzerod}{u^{(0)}_{\delta}}
\newcommand{\umd}{u^{(-1)}_{\delta}}
\newcommand{\ukmd}{u^{(k-1)}_{\delta}}
\newcommand{\ukmmd}{u^{(k-2)}_{\delta}}
\newcommand{\uimd}{u^{(i-1)}_{\delta}} 
\newcommand{\uimmd}{u^{(i-2)}_{\delta}}
\newcommand{\ukd}{u^{(k)}_{\delta}}

\newcommand{\ukzero}{u^{(k)}_{0}} 
\newcommand{\ukmzero}{u^{(k-1)}_{0}}
\newcommand{\ukmmzero}{u^{(k-2)}_{0}}
\newcommand{\uimzero}{u^{(i-1)}_{0}}
\newcommand{\uimmzero}{u^{(i-2)}_{0}}

\newcommand{\rkd}{r^{(k)}_{\delta}}

\newcommand{\rkzero}{r^{(k)}_{0}}

\newcommand{\ld}{\lambda_{\delta}}

\newcommand{\kd}{k_{\delta}}
\newcommand{\wkdd}{w^{(\kd)}_{\delta}}
\newcommand{\wldkdd}{w^{(\ld/\kd)}_{\delta}}
\newcommand{\vkdd}{v^{(\kd)}_{\delta}}
\newcommand{\ukdd}{u^{(\kd)}_{\delta}}
\newcommand{\zkdd}{z^{(\kd)}_{\delta}}

\newcommand{\vkdzero}{v^{(\kd)}_{0}}

\newcommand{\zk}{z^{(k)}}

\newcommand{\zdk}{z_{\delta}^{(k)}}

\newcommand{\Km}{\underline{k}}
\newcommand{\KM}{\overline{k}}

\newcommand{\Kd}{ \left[ \Km,~ \KM \right] }

\newcommand{\Tm}{\underline{t}}
\newcommand{\TM}{\overline{t}}

\newcommand{\Tw}{T_{w}}
\newcommand{\Twsol}{T_{\wsol}}
\newcommand{\XTwsol}{X_{\Twsol}}
\newcommand{\PTwsol}{P_{\Twsol}}
\newcommand{\Mwsol}{\calM_{\wsol}}

\newcommand{\MIR}{M_{\mathrm{DGD}}}

\newcommand{\rhoMIR}{\rho_{_{\MIR}}}

\newcommand{\cd}{c_{\delta}}
\newcommand{\td}{t_{\delta}}

\newcommand{\wdt}{w_{\delta}(t)}
\newcommand{\vdt}{v_{\delta}(t)}

\newcommand{\wdtd}{w_{\delta}(\td)}

\newcommand{\vzero}{v_{0}}

\newcommand{\dvdt}{\dot{v}_{\delta}(t)}

\ifnum\arxivorIP=1

\documentclass[12pt]{iopart}
\usepackage{graphicx}
\usepackage[caption=false]{subfig}
\usepackage{amsfonts}
\usepackage{amsthm}
\usepackage{algorithm}
\usepackage{algpseudocode}

\usepackage{fullpage}
\usepackage{adjustbox}
\usepackage{color}
\usepackage{relsize}
\expandafter\let\csname equation*\endcsname=\relax
\expandafter\let\csname endequation*\endcsname=\relax
\usepackage{amsmath,amssymb,amsthm}
\newtheorem{theorem}{Theorem}[section]

\newtheorem{definition}{Definition}[section]
\newtheorem{remark}{Remark}[section]
\newtheorem{corollary}{Corollary}[section]

\newtheorem{lemma}{Lemma}[section]

\begin{document}

\title[Model Consistency of the Iterative Regularization of DGD]{Model Consistency of the Iterative Regularization of Dual Ascent for Low-Complexity Regularization}

\author{Jie Gao$^1$, Cesare Molinari$^2$, Silvia Villa$^3$, Jingwei Liang$^1$}

\address{$^1$ School of Mathematical Sciences and Institute of Natural Sciences, Shanghai Jiao Tong University, Shanghai, People’s Republic of China}
\address{$^2$ MaLGa, DIMA, Dipartimento di Eccellenza 2023-2027, Universit\`{a} degli Studi di Genova, Genoa, Italy}
\address{$^3$ MaLGa, Dipartimento di Matematica, Universit\`{a} degli Studi di Genova, Genoa, Italy}
 
\ead{ \{zjgaojie,\ jingwei.liang\}@sjtu.edu.cn, cecio.molinari@gmail.com, silvia.villa@unige.it}
\vspace{10pt}
\begin{indented}
\item[]\today
\end{indented}

\begin{abstract}
\myabstract
\end{abstract}

\submitto{\IP}
%
%

\fi 

\ifnum\arxivorIP=0
\documentclass[11pt,a4paper,reqno]{article} 
\usepackage{mysty}

\usepackage{amssymb}
\usepackage[thinlines]{easytable}
\usepackage{relsize}

\usepackage{fullpage}
\usepackage{adjustbox}
\flushbottom

\usepackage[normalem]{ulem}

\usepackage{multirow}
\usepackage{graphicx} 

\setlength{\parindent}{12pt}

\SetLabelAlign{parright}{\parbox[t]{\labelwidth}{\raggedleft{#1}}}
\setlist[description]{style=multiline,topsep=4pt,align=parright}

\makeatletter
\let\reftagform@=\tagform@
\def\tagform@#1{\maketag@@@{(\ignorespaces\textcolor{black}{#1}\unskip\@@italiccorr)}}
\newcommand{\iref}[1]{\textup{\reftagform@{\tcr{\ref{#1}}}}}
\makeatother

\usepackage[noindentafter]{titlesec}
\usepackage{bold-extra}
\usepackage{anyfontsize} 

\titlespacing\section{0pt}{11pt plus 4pt minus 2pt}{6pt plus 2pt minus 2pt}
\titlespacing\subsection{0pt}{10pt plus 4pt minus 2pt}{4pt plus 2pt minus 2pt}
\titlespacing\subsubsection{0pt}{8pt plus 4pt minus 2pt}{4pt plus 2pt minus 2pt}
\titlespacing\paragraph{0pt}{6pt plus 4pt minus 2pt}{6pt plus 2pt minus 2pt}


\begin{document}

\setlength{\abovedisplayskip}{4.5pt}
\setlength{\belowdisplayskip}{4pt}

\title{Model Consistency of the Iterative Regularization of Dual Ascent for Low-Complexity Regularization}
\author{
Jie Gao\thanks{School of Mathematical Sciences and Institute of Natural Sciences, Shanghai Jiao Tong University, Shanghai, China. E-mail:  \{zjgaojie,\ jingwei.liang\}@sjtu.edu.cn} \and
 		Cesare Molinari \thanks{MaLGa, DIMA, Dipartimento di Eccellenza 2023-2027, Universit\`{a} degli Studi di Genova, Genoa, Italy. E-mail: cecio.molinari@gmail.com; silvia.villa@unige.it} \and
        Silvia Villa\footnotemark[2] \and
		Jingwei Liang\footnotemark[1]
        }
\date{}
\maketitle

\begin{abstract}
\myabstract
\end{abstract}

\begin{keywords}
\mykeywords
\end{keywords}

\begin{AMS}
{ 47A52 $\cdot$ 49J52 $\cdot$ 60H50 $\cdot$ 65K10 $\cdot$ 90C31 }
\end{AMS}

\fi 

\section{Introduction}

Linear inverse problems are widely encountered in many fields through science and engineering, such as signal processing, compressed sensing, (medical) image processing, and remote sensing.
In (linear) inverse problems, the target of interests, denoted by $\wsol$, is not directly accessible. It is  only observed indirectly through a linear measurement, typically affected by noise.  
In the finite dimensional setting, let $X \in \bbR^{n\times p}$ represent the linear mapping and $\yzero\in\calY$ the noise-free observation. Then,
\begin{equation} 
\yzero = X \wsol,
\notag\end{equation}
Retrieving $\yzero$ is referred to as the {\it forward problem}. Due to imperfections in the observation process, the observed data is often corrupted by noise, leading to the following model:
\begin{equation} \label{eq:yd}
\yd = X \wsol +\varepsilon,
\end{equation}
where $\varepsilon$ is additive white Gaussian noise with noise level $\delta$. 

Inverse problems aim to recover or approximate $\wsol$ from the observation data. In the {\it noiseless} case,  the problem can be described as 
\begin{equation}\label{eq:feasibility}
\find w\in\calW \quad{\rm such~ that}\quad X w= \yzero.
\end{equation} 
However, the above problem is in general ill-posed mostly due to the poor conditioning or degeneracy of the forward operator $X$. 
The situation becomes even more challenging in the presence of noise.

\subsection{Regularization} 
The ground truth often exhibits certain low-complexity structural properties referred to as the model which can be explicitly promoted using a regularization function 
$R(w)$, also called regularizer.
Common structural priors include: sparsity, induced by the $\ell_1$-norm \cite{candes2006robust}; group sparsity, induced by the $\ell_{1,2}$-norm \cite{huang2010benefit}, gradient sparsity corresponding to the total variation~\cite{rudin1992nonlinear}) regularizer, and low rank favored by nuclear norm \cite{candes2012exact}. For a broader discussion, see \cite{benning2018modern}.
By incorporating regularization, problem \eqref{eq:feasibility} is reformulated as:
\begin{equation}\label{P}\tag{$P$}
\min_{w\in\calW} R(w) \quad{\rm such~that}\quad \calX w = \yzero.
\end{equation}
More generally, one may introduce a data-fit function $\ell:\calY\times \calY \rightarrow \bbR \cup \{ +\infty \}$, 
\begin{equation}    
\min_{w\in\calW} R(w)  \quad{\rm such~that}\quad  w\in\arg\min \big\{\ell(\calX w,\yzero)\big\}.\tag{$\Pzero$}\label{P0}
\end{equation}
Problem \eqref{P0} recovers \eqref{P} when $\ell$ is the indicator function of the set $\{ w\in\calW \mid Xw=\yzero\}$. 

In the noisy setting, one can consider the relaxed constraint
\begin{equation}    
\min_{w\in\calW} R(w)  \quad{\rm such~that}\quad  \ell(\calX w,\yd) \leq \delta , \tag{$\widehat{P}_{\delta}$}\label{Pcd}
\end{equation}
where $\delta$ is the noise level. However, this formulation depends on knowing or estimating the noise level 
$\delta$, which is often difficult. A more tractable approach incorporates the constraint into the objective with a weighting parameter 
$\lambda>0$, yielding the variational regularization model:
\begin{equation}
\min_{w\in\calW} \Big\{ \pldw:=R(w)+\frac{1}{\lambda}\ell(\calX w,\yd) \Big\} . \tag{$\Pld$}\label{Pld}
\end{equation}
The regularization parameter $\lambda$ can be selected via validation criteria such as the discrepancy principles \cite{morozov1966solution}, cross-validation  \cite{golub1997generalized} and SURE~\cite{stein1981estimation,deledalle2014stein}. Recently, data-driven approaches have also been designed to determine the proper regularization parameter~\cite{ArrMaa19,chirinos2024learning}. 
This approach is known as variational regularization. Once a formulation is chosen, an appropriate numerical scheme is applied and ran until convergence, often for different values of the parameter $\lambda$.

An alternative approach is {\it iterative regularization},  which achieves regularization through early stopping of a suitable numerical procedure. Specifically, for noisy data 
$\yd$, one may consider problem \eqref{P0} with 
$\yzero$ replaced by  $\yd$:
    \begin{equation}
    \min_{w\in\calW} R(w)\quad{\rm such~that}\quad w\in\arg\min\big\{\ell(\calX w,\yd)\big\}.\tag{$\Pd$}\label{Pd}
    \end{equation}
Assuming the constraint set is non-empty, this problem can be addressed with an iterative algorithm that generates a sequence $\{\wkd\}_{k\in\bbN}$. 
Since the constraint involves noisy data, the sequence may eventually overfit to the noise. Thus, a suitable stopping time  $\kd$ must be chosen to avoid overfitting. 
In this framework, regularization is achieved via early stopping.

In practice, stopping criteria are often selected using the same validation principles used in variational regularization. However, there are key differences: on the one side, the choice of $\ld$ in the variational regularization approach is generally computationally expensive while $\kd$ controls at the same time the accuracy of the solution and the computational cost. For large scale problems, iterative regularization is more efficient because it avoids solving various optimization problems. On the downside, theoretical analysis (e.g., convergence and rates) for iterative regularization is more complex, as it depends on the specific algorithm and the solution is not the result of a well-defined optimization problem.

\subsection{Model Consistency}

The premise of regularization is that the ground truth $\wsol$ possesses a low-complexity structure, often referred to as its ``model''. Such structure can often be described by a smooth manifold $\Mwsol$, such that $\wsol \in \Mwsol$. This includes commonly encountered low-complexity structures such as (group) sparsity, low rank, and analysis sparsity (e.g., piecewise constant signals).

At the same time, the outcome of the aforementioned regularization approaches is designed to promote solutions with similar low-complexity structures. 
Denote the output of the variational or iterative regularization method by $\wldkdd$. A natural question is whether $\wldkdd$ exhibits the same structure as 
$\wsol$, i.e., whether $\wldkdd \in \Mwsol$ also holds. This property is referred to as \textit{model consistency}.

Model consistency reflects the robustness of the solution with respect to perturbations. In the context of linear inverse problems and variational regularization, model consistency has been studied in~\cite{vaiter2015model,vaiter2017model,fadili2019model}. Specifically, when the data-fit function is quadratic, namely $\ell(\calX w, \yd) = \frac{1}{2} \| \calX w - \yd \|^2$, it is shown that if the noise in $\yd$ is sufficiently small, and the regularization parameter $\lambda$ is properly chosen, then the solution to \eqref{Pld} satisfies the model consistency property.
However, the model consistency of iterative regularization remains unaddressed in the existing literature. In this paper we do a first step in this direction.

\subsection{Related work} 
\paragraph{Iterative regularization}
The study of the iterative regularization properties of gradient descent, or the Landweber iteration, dates back to the 1950's. Classical results show that gradient descent applied to least-squares problems, when initialized at zero, converges to the minimal-norm solution, and that early stopping behaves similarly to Tikhonov regularization~\cite[Chapter 6]{engl1996regularization}. 
Similar results for stochastic gradient descent are provided in~\cite{jin2018regularizing}. 
For strongly convex regularization functions, iterative regularization of dual gradient descent, linearized Bregman iterations and mirror descent are studied in~\cite{burger2007error,rosasco2015learning,villa2023implicit}, with corresponding stability and convergence estimates established. 
For merely convex regularization, methods such as Bregman operator splitting, linearized/preconditioned ADMM, and primal-dual methods?including data-driven variants?have been analyzed in~\cite{molinari2021iterative,molinari2024iterative,vega2024fast}.
When using data-fit functions other than least squares, the aforementioned methods do not directly apply. One approach to handle this is the use of diagonal strategies~\cite{garrigos2018iterative,calatroni2021accelerated}, which combine an optimization algorithm with a sequence of approximations \(\ell_k\) to the original problem, updated at each iteration.

For only convex regularization, a technique known as exact regularization~\cite{friedlander2008exact,yin2010analysis} exists. This involves solving the problem
\begin{equation}\label{Ralpha}
\min_{w\in\bbR^p} \Big\{ \Ra (w):=R(w)+\frac{\alpha}{2} \|w\|^2 \Big\} \qstq\calX w=\yzero,
\end{equation} 
where there exists small enough $\alpha>0$ such that the minimizer of the perturbed problem coincides with that of the original problem~\eqref{P}. In this way, a strongly convex regularizer \(\Ra\) is obtained by adding a quadratic term.

\paragraph{Model consistency}
Using the framework of partial smoothness~\cite{lewis2002active,lewis2024identifiability}, it can be shown that, in variational regularization, the model can be identified when $\delta$ is sufficiently small and an appropriate $\ld$ is chosen~\cite{vaiter2015model,vaiter2017model,fadili2019model}. This implies  model consistency. 
A related topic is the analysis of local linear convergence of first-order optimization methods.  For relevant results involving first-order optimization algorithms, see for instance \cite{liang2014local,liang2017activity,molinari2019convergence,nutini2022let,klopfenstein2024local} and  references therein.

\subsection{Contribution of this work}

In this paper, motivated by recent developments regarding partial smoothness \cite{qin2025partial}, we study the model consistency property of iterative regularization. Under the assumption that the low-complexity  regularization function is partly smooth and strongly convex, we analyze problem~\eqref{Pd} solved by the dual gradient descent (DGD) method, as studied in~\cite{villa2023implicit}, and establish the following results:
\begin{itemize}
    \item We show that if the noise level is sufficiently small and an appropriate early stopping time is selected, the resulting solution from iterative regularization satisfies the model consistency property; see Theorem~\ref{thm:mc_DGD}~(i). Moreover, we prove that for a fixed, sufficiently small noise level, there exists an interval of stopping times over which model consistency holds; see Theorem~\ref{thm:mc_DGD}~(ii). We also extend this result to the accelerated dual gradient descent method.   
     
     \item We consider the continuous dynamical system corresponding to the chosen discrete algorithm \cite{ApiMol23}, and show that the model consistency property still holds in this setting.
     
    \item For all the points satisfying model consistency, we provide a finer local characterization of their distance, which is called regularization error, to the true solution. In Theorem \ref{thm:reg-error}, we show that the behavior of the error is controlled by a factor which is strictly smaller $1$ under a restricted injectivity condition; see \eqref{INJ}.
    
\end{itemize}
Finally, we present numerical results on $\ell_1$-norm, one dimensional total variation (1d-TV), $\ell_{1,2}$-norm, and nuclear norm. These experiments confirm our theoretical findings.

\paragraph{Paper Organization}
The remainder of the paper is organized as follows: Section~\ref{sec:mb} provides mathematical background, including basic concepts from convex analysis and the definition of partial smoothness. In Section~\ref{sec:ir}, we present the problem formulation along with the dual gradient descent (DGD) and accelerated dual gradient descent (ADGD) algorithms. Section~\ref{sec:mcir} contains the main theoretical analysis of model consistency for iterative regularization. Section~\ref{sec:llc} presents results on local analysis of the regularization error. Finally, Section~\ref{sec:ne} provides numerical examples demonstrating showing how model consistency and regularization error behave in practice.

\section{Mathematical background}\label{sec:mb}
In this section, we collect the necessary background material, that serves in our subsequent theoretical analysis. 

\subsection{Notation and definitions}

We refer the reader to~\cite{beck2017first,bauschke2017convex} for more details on the material presented below. 
Let $\calW$ be a Euclidean space equipped with the inner product $\langle \cdot, \cdot \rangle$ and the associated norm $\| \cdot \|$. 
We denote by $\Id$ the identity matrix. Given $w \in \calW$ and $\epsilon > 0$, let $\calB(w, \epsilon)$ be the open ball centered at $w$ with radius $\epsilon$.

Given a non-empty, closed, and convex set $S \subseteq \calW$, we denote by $\mathrm{int}(\cdot)$ and $\mathrm{bdy}(\cdot)$ the interior and boundary of a set, respectively.

The relative interior of $S$ is defined by 
$$
\ri(S)=\Ba{w\in S \mid \calB(w,\epsilon)\cap \aff S\subseteq S\ \text{for some}\ \epsilon>0 },
$$
where $\aff S$ is the smallest affine subspace containing  $S$. 
Given a point $w\in\calW$, the distance of $w$ from $S$ and the projection of $w$ onto $S$ are defined as follows
\[
{\rm dist}(w, S) = \inf_{w'\in S} \|w-w'\| 
\qandq
P_S(w)={\arg \min}_{w'\in S} ~ \|w-w'\| .
\]

We denote by $\lsc$ the set of proper, convex, and lower semicontinuous functions from $\calW$ to $\bbR\cup\{+\infty\}$. 
We say that $R\in\lsc)$ is $\sigma$-strongly convex if $R-\frac{\sigma}{2}\|\cdot\|^2$ is convex and $R$ is $L$-smooth if it is differentiable and $\nabla R$ is $L$-Lipschitz continuous. The subdifferential 
of $R\in\lsc$ at a point $w$  is defined as
$$\partial R(w)=\Big\{ u\in\calW \mid R(w')-R(w)-\langle u,w'-w\rangle\geq0,\forall w'\in\calW\Big\},$$
if $w\in\mathrm{dom}R$ and is the empty set otherwise. 
Any element in $\partial R(w)$ is called a subgradient of $R$ at $w$. The conjugate function $R^*$ of $R\in\lsc$ is defined by
$$R^*(u)=\sup_{w\in\calW}\big\{\langle u,w\rangle-R(w)\big\}.$$ 
If $R\in\lsc$, then $R^*\in\lsc$. 
Moreover, if $R\in\lsc$ is $\sigma$-strongly convex, with $\sigma>0$, the conjugate function $R^*$ is $\frac{1}{\sigma}$-smooth. 
Given $\gamma>0$, the proximity operator of $R$ is defined by
$$
\prox_{\gamma R}(w)
={\arg\min}_{w'\in\calW}\bBa{ R(w')+\qfrac{1}{2\gamma}\|w'-w\|^2 },
$$ and is firmly nonexpansive.
\subsection{Partial smoothness}
Given $S\subseteq\calW$ a closed convex set, the smallest linear subspace of $\calW$ that contains $S-x, \forall x\in S$ is denoted by 
$\operatorname{span}(S)$. 
Let $\calM\subset\bbR^p$ be a $C^2$-manifold around $w\in\bbR^p$, and denote $\calT_\calM(w)$ the tangent space of $\calM$ at $w$. 
For any vector $w\in\bbR^p$ with $\partial R(w) \neq \emptyset$, the model tangent subspace is defined by
$$
\Tw = \big(\operatorname{span}(\partial R(w))\big)^{\perp}.
$$

The following definition of partial smoothness is adapted from~\cite{lewis2002active} to the convex setting.
\begin{definition}[Partly smooth function]
Let $R\in\Gamma_0(\bbR^p)$ and $\wsol\in\bbR^p$ such that $\partial R(\wsol)\not= \emptyset$. $R$ is said to be partly smooth at $\wsol$ relative to a set $ \Msol\subset \bbR^p$ if 
\begin{itemize}
\item \textbf{(Smoothness)} $\Msol$ is a $C^2$ manifold around $\wsol$ and $R|_{\Msol}$ is $C^2$  near $\wsol$;
\item \textbf{(Sharpness)} The tangent space $\calT_\Msol(\wsol)$ coincides with the model tangent subspace $\Twsol$;
\item \textbf{(Continuity)} $\partial R$  is continuous near $\wsol$ along $\Msol$.  
\end{itemize}
\end{definition}

\begin{remark}
When $R$ is partly smooth at \( \wsol \) relative to a manifold \( \Msol \), the function, $\Ra =R(w)+\frac{\alpha}{2}\|w\|^2, \alpha>0$ is also partly smooth at $\wsol$ relative to the same manifold, according to the smooth perturbation rule for partial smoothness \cite[Corollary 4.7]{lewis2002active}. 
\end{remark}

In Table~\ref{table:examples-psf} below, we summarize several popular examples of partly smooth functions that are widely used in inverse problems.

\begin{table}[!htbp]
\caption{Examples of partly smooth functions: $\supp(w):=\{i:w_{i}\not=0\}$; $\DDIF$ stands for the finite differences operator; for $\ell_{1,2}$-norm, $w_{_\calI}$ is the restriction of $w$ to the entries indexed in $\calI$, $\cup_{i=1}^G \calI_i=\{1,\cdots,p\}$, $\supp_{_\calI}(w):=\{i:w_{_{\calI_i}}\not=0\}$.}
\label{table:examples-psf}
\begin{center}
\begin{tabular}{|c|c|c|}
\hline
Function &Expression& Partial smooth manifold\\
\hline
$\ell_1$-norm&$\sum_{i=1}^p|w_{i}|$ & $\Msol=\Twsol=\{w\in\bbR^p \mid \supp(w)\subset \supp(\wsol)\}$\\
\hline
TV semi-norm&$\|\DDIF w\|_1$& $\Msol=\Twsol=\{w\in\bbR^p \mid \supp(\DDIF w)\subset \supp(\DDIF\wsol)\}$\\
\hline
$\ell_{1,2}$-norm&$\sum_{i=1}^m\|w_{_{\calI_i}}\|$ & $\Msol=\Twsol=\{w\in\bbR^p \mid \supp_{_{\calI}}(w)\subset \supp_{_{\calI}}(\wsol)\}$\\
\hline
Nuclear norm &$\sum_{i=1}^r \sigma_i(w)$& $\Msol=\{w\in \bbR^{n\times p} \mid \rank(w)=\rank(\wsol)\}$\\
\hline
\end{tabular}
\end{center}
\footnotesize
Remark: For the first three partly smooth functions the manifold $\Msol$ is affine, while for the nuclear norm it is not. Moreover, since the first two are polyhedral functions, the subdifferential is locally constant around $\wsol$ along $\wsol+\Twsol$.
\end{table}

\paragraph{Finite identification}
A crucial property of partial smoothness is the \emph{identifiability} of the smooth manifold. This means that if there exist sequences \( w^{(k)} \) and \( z^{(k)} \in \partial R(w^{(k)}) \) such that
\[
\operatorname{dist}(z^\dagger, \partial R(w^{(k)})) \to 0 \quad \text{for some } z^\dagger \in \operatorname{ri}(\partial R(w^\dagger)),
\]
then \( w^{(k)} \in \mathcal{M}_{w^\dagger} \) for all sufficiently large \( k \). This property is called \emph{finite identification}~\cite{hare2004identifying}.  
Below, we provide a recent generalization of this result from~\cite{hare2004identifying}, which does not require convergence of the sequence.

\begin{lemma}[{\cite[Proposition 4, Theorem 4]{qin2025partial}}]\label{lem:1}
Let  $R \in \lsc$ be  partly smooth at $\wsol$ relative to a $C^2$-manifold $\Msol \subset \bbR^p$. For any $\epsilon >0$, define the local union as
\[
\calU_{\epsilon} = {\mathsmaller \bigcup}_{w\in \Msol\cap \calB(\wsol,\epsilon)} \big( w + \partial R(w) \big)  .
\]
Then  ${\mathrm{span}}{(\calU_{\epsilon})}=\bbR^p$, and 
\begin{enumerate}
\item If $z^\dagger \in \ri(\partial R(\wsol))$, then $\wsol + z^\dagger \in {\rm int}\Pa{\calU_{\epsilon}} $. 

\item If there is a sequence $\{(w^{(k)}, z^{(k)})\}_{k\in\bbN}$ such that $\zk\in\partial R(\wk)$ and 
\[
\limsup_{k\to\infty} \| (w^{(k)} + z^{(k)}) - (\wsol + z^\dagger) \| < {\rm dist} \Pa{\wsol +  z^\dagger, ~ \bdy(\calU_{\epsilon})}  ,
\]
then for all $k$ large enough we have $\wk \in \Msol$. 
\end{enumerate}
\end{lemma}

\begin{remark}
Lemma \ref{lem:1} is a summary of the result of \cite[Proposition 4 and Theorem 4]{qin2025partial}. Note that in \cite{qin2025partial}, the value of $\epsilon$ needs to be small enough, while in our case we do not have this restriction, the reason is that we are in the convex setting while \cite{qin2025partial} is not. However, $\epsilon$  may affect the value of ${\rm dist} \Pa{\wsol +  z^\dagger, ~ \bdy(\calU_{\epsilon})}$ when it is too small, and we refer to \cite[Figure 1]{qin2025partial} for illustration. 
\end{remark}

\paragraph{Riemannian gradient and Hessian}

When $\Msol$ is affine, according to \cite{liang2017activity} the Riemannian gradient and Hessian of $R$ at $w\in\Mwsol$ read
\begin{equation}\label{eq:riemannain_gradient_hessian}
    \nabla_{\Mwsol} R(w) = \PTwsol \partial \widetilde{R}(w) 
    \qandq
    \nabla^2_{\Mwsol} R(w)  = \PTwsol \nabla^2 \widetilde{R}(w) \PTwsol ,
\end{equation}
where $\widetilde{R}$ is a $C^2$-smooth representation of $R$ along $\Mwsol$, i.e. a $C^2$-smooth function $\widetilde{R}$ on $\bbR^n$ that agrees with $R$ over $\Mwsol$. 

We have the following Riemannian Taylor expansion which is a simplified version of \cite[Lemma B.2]{liang2017activity}. 

\begin{lemma}\label{lem:taylor}
Let $\Mwsol$ be affine and two points $w_1, w_2 \in \Mwsol$ which are close, the Riemannian Taylor expansion of $R\in C^2(\Mwsol)$ around $w_1$ reads
\[
\nabla_{\Mwsol} R(w_2)
= \nabla_{\Mwsol} R(w_1) + \nabla^2_{\Msol} R(w_1) (w_2 - w_1) + o(\norm{w_2-w_1}) .
\]
\end{lemma}

\section{Iterative Regularization}\label{sec:ir}

In this part, we briefly recall the iteration regularization of dual gradient descent method and accelerated dual gradient descent method studied in \cite{garrigos2018iterative,villa2023implicit}. 
For problem \eqref{Pd}, suppose $\yd$ belongs to the range of $\calX$, then we can specify the constraint as $\calX w = \yd$. Regarding the regularization, strongly convex functions is considered as in \eqref{Ralpha}, leading to the following optimization problem
\begin{equation}
\min_{w\in\bbR^p} \Ra (w) \qstq \calX w = \yd .
\tag{$\overline{P_\delta}$}\label{Pdelta0}   
\end{equation}

\paragraph{Dual problem}
By Fenchel-Rockafellar duality, the dual problem of \eqref{Pdelta0} is
\begin{equation}
\min_{v\in\bbR^n} \bBa{ \phi(v) := \Ra ^*(-X^{\top}v)+\langle \yd,v\rangle } , \tag{$\overline{D_\delta}$}\label{Ddelta0}
\end{equation}
where $\Ra ^*(-X^{\top}v)$ is the conjugate of $\Ra$, which is given by the Moreau envelope of $R^*$ \cite[Proposition 14.1]{bauschke2017convex}
$$ 
\Ra ^*(\cdot)= \inf_{w} \Ba{ R(w)+\frac{1}{2\alpha}\|w- {\cdot}\|^2 }.
$$
By \cite[Proposition 12.30]{bauschke2017convex} $\Ra^*$ is differentiable  with  $\nabla \Ra ^*=\alpha^{-1} (\Id-\prox_{\alpha R^*}),
$
so that, thanks to the Moreau identity \cite[Theorem 14.3(ii)]{bauschke2017convex}  we have
\begin{equation}
\nabla \Ra ^*=\prox_{\alpha^{-1}R} (\alpha^{-1}x)
\notag\end{equation}
Moreover, since $\nabla \phi(v) = - X \prox_{\alpha^{-1}R} \Pa{ - \alpha^{-1} X^{\top}v } + \yd$, $\nabla\phi$ is $\frac{\norm{X}^2}{\alpha}$-Lipschitz continuous. 
Hence we can apply gradient descent method to solve \eqref{Ddelta0}, as considered in \cite{villa2023implicit}.

\paragraph{Dual gradient descent}

The dual gradient descent (DGD) method for solving (\ref{Pdelta0}) is described below.

\begin{algorithm}
\caption{Dual gradient descent (DGD)}
\label{alg:DGD}
\begin{algorithmic}[1]
\Require $\vzerod=0\in\bbR^n,\gamma=\alpha\|X\|^{-2},X,\yd,R,\alpha$
\Ensure $\{\wkd\}$
\For{$k = 0,1,2,\dotsm$}
\State $\zkd= - X^{\top} \vkd $ \vspace{1.5mm}
\State $\wkd=\prox_{\alpha^{-1}R}\Pa{ \alpha^{-1} \zkd } $ \vspace{1.5mm}
\State $\vkpd=\vkd+\gamma(X\wkd-\yd)$
\EndFor
\end{algorithmic}
\end{algorithm}

The iterate $\zkd$ is a subgradient of $\Ra$ at $\wkd$, as from the optimality condition of proximal operator we have
\[
\qfrac{1}{\alpha} \zkd - \wkd \in \qfrac{1}{\alpha} \partial R(\wkd) 
\quad\Longleftrightarrow\quad
\zkd \in \partial \Ra(\wkd) .
\]
The iteration can be written only in terms of $\vkd$ as
\begin{equation}\label{eq:dgd-vkd}
\vkpd = \vkd - \gamma \nabla \phi(\vkd)  ,
\end{equation}
which is a gradient descent step for the dual problem \eqref{Ddelta0}. According to \cite{bauschke2017convex,liang2016convergence}, we have the following convergence result.

\begin{lemma}\label{lem:convergence-vk}
For the DGD Algorithm~\ref{alg:DGD}, there exists a $\vdsol \in {\rm arg\min}(\phi)$ such that $\vkd \to \vdsol$ as $k\to\infty$. Moreover, there exists a constant $C>0$ 
\[
\norm{\vkd-\vkmd} \leq \qfrac{C}{\sqrt{k}} .
\]
\end{lemma}
In iteration regularization, we are interested in the relation between the early stopping point of $\wkd$ and the ground truth $\wsol$. We impose the following assumption and notion
\begin{itemize}
    \item The ground truth $\wsol$ is the unique solution of \eqref{Pdelta0} for $\delta=0$.

    \item The early stopping time is denoted as $\kd$, and the corresponding point is $\wkdd$. 
\end{itemize}
From \cite{villa2023implicit}, we have the following result. 

\begin{lemma}[{\cite[Theorem 4.1]{villa2023implicit}}]\label{lem:DGD}
Let $\delta\in]0, 1[$, and consider the DGD Algorithm~\ref{alg:DGD}. 
Suppose there exists $\bar{v} \in \bbR^n$ such that $-\calX^\top \bar{v} \in \partial \Ra(\wsol)$. 
Set $ a=2\|X\|^{-1}$ and $b = \|X\|\|\vsol\|\alpha^{-1}$. 
Given any $c\geq \delta$ and $k_{\delta}\in\big\{\lfloor c\delta^{-1}\rfloor,\cdots,2\lfloor c\delta^{-1}\rfloor\big\}$, there holds 
$$
\|\wkdd-\wsol\|
\leq \Pa{ a(c^{1/2}+1)+bc^{-1/2} }\delta^{1/2},
$$
Let $\cd>0$ be such that $\kd=\cd\delta^{-1}$.
\end{lemma}

The above stability result implies that the early stopping point $\wkdd$ converges to $\wsol$ as $\delta\to 0$ at the rate of $\delta^{1/2}$. Moreover, the amount of computations is related to the noise level, in particular $\kd\to\infty$ as $\delta \to 0$. 

In \cite{villa2023implicit}, they also consider an accelerate version of DGD, denoted as ADGD, namely Nesterov's accelerated gradient method \cite{nesterov83} applied to DGD; see the algorithm below.

\begin{algorithm}[H]
\caption{Accelerated dual gradient descent (ADGD)}\label{alg:ADGD}
\begin{algorithmic}[1]
\Require $\vzerod=\umd=\uzerod=0\in\bbR^n,\gamma=\alpha\|X\|^{-2}$, $X,\yd,R,\alpha$ and $\theta>2$ 
\Ensure $\{\wkd\}$
\For{$k = 0,1,2,\dotsm$}
    \State $\rkd=\prox_{\alpha^{-1}R}(-\alpha^{-1}X^{\top} \vkd)$   	\vspace{1.5mm}
    \State $\ukd=\vkd+\gamma(X\rkd-\yd)$								\vspace{1.5mm}
    \State { $\vkpd=\ukd+\sfrac{k-1}{k+\theta}(\ukd-\ukmd)$	}\vspace{1.5mm}
    \State {$\wkd=\prox_{\alpha^{-1}R}(-\alpha^{-1}X^{\top}\ukd)$}
\EndFor
\end{algorithmic}
\end{algorithm}

The above iteration is slightly different from the scheme considered in \cite{villa2023implicit} in terms of the inertial parameter $\sfrac{k-1}{k+\theta}$. With this choice we still have $O(1/k^2)$ convergence rate for the dual objective, dual sequence \cite{chambolle2015convergence}. For the original  (FISTA like) choice of the parameters, one could  consider the modified accelerated scheme in \cite{liang2022improving}. 

\begin{lemma}\label{lem:convergence-adgd}
    For the ADGD Algorithm~\ref{alg:ADGD}, let $\theta>2$, then there exists a $\vdsol \in {\rm arg\min}(\phi)$ such that $\ukd \to \vdsol$ as $k\to\infty$. Moreover, 
    \[
    \sum_{k=1}^{\infty} k \norm{\ukd-\ukmd}^2 < +\infty  .
    \]
\end{lemma}

The above result implies that $\norm{\ukd-\ukmd} = o(1/k) $. Unlike DGD, ADGD is not a monotone scheme in the sense $\norm{\ukd-\vsol}\leq\norm{\ukmd-\vsol}$ is not necessarily true. Also note that the result holds for any $\theta > 2$.

\begin{lemma}[{\cite[Theorem 4.2]{villa2023implicit}}]\label{lem:ADGD}
Let $\delta\in]0, 1[$. Let $\{\wkd\}$ be the sequence generated by Algorithm \ref{alg:ADGD}. 
Suppose there exists $\bar{v} \in \bbR^n$ such that $-\calX^\top \bar{v} \in \partial \Ra(\wsol)$. 
Set $ a=4\|X\|^{-1}$ and $b = 2\|X\|\|\vsol\|\alpha^{-1}$, where $\vsol$ is a solution of the dual problem. Then 
$$
\|\wkd-\wsol\|\leq a k \delta + b k ^{-1}.
$$ 
In particular, choosing $\kd = \lceil c\delta^{-1/2}\rceil$ for some $c>0$, then there holds
$$
\| \wkdd-\wsol\|\leq \Pa{ a(c+1) + bc^{-1}} \delta^{1/2}.
$$
\end{lemma}

\section{Model Consistency of Iterative Regularization}\label{sec:mcir}
In this section, we present our main result: the model consistency of iterative regularization.
From the above results, for sufficiently small $\delta>0$, the early stopped iterate $\wkdd$ will be sufficiently close to $\wsol$. If we can further control the dual variable, then Lemma \ref{lem:1} can be applied to demonstrate that $\wkdd$ is model consistent. 

Regarding the dual variable, similarly to the work of \cite{vaiter2017model}, we  rely on the following {\it nondegenerate source condition}  for the unique solution  $\wsol$  of \eqref{Pdelta0} when $\delta=0$. 

\begin{definition}[Nondegenerate source condition]
For the ground truth $\wsol$, let $\vsol$ be a dual solution such that:
\begin{equation}
\zsol := - X^\top \vsol \in\ri\Pa{ \partial \Ra(\wsol)} . \tag{$\overline{\SC}_{\wsol}$}\label{NSC} 
\end{equation}
\end{definition}
When $R$ is partly smooth at $\wsol$ relative to $\Msol$, so is $\Ra$. For any $\epsilon > 0$, define 
\begin{equation}\label{eq:Ue}
\calU_{\epsilon} = {\mathsmaller \bigcup}_{w\in \Msol\cap \calB(\wsol,\epsilon)} \Pa{ w + \partial \Ra(w) }. \end{equation}  
According to Lemma \ref{lem:1}, there holds 
$
\wsol + \zsol \in {\rm int}\Pa{ \calU_{\epsilon} } . 
$
Denote 
$
r:=\dist \Pa{ \wsol+\zsol,\bdy(\calU_{\epsilon}) }
$
the distance of $\wsol + \zsol$ to the boundary of $\calU_{\epsilon}$. 
Back to the iterates of DGD, define $\zkdd= - X^\top \vkdd$. If the distance bound
\begin{equation}\label{eq:distance}
    \norm{(\wkdd+\zkdd)-(\wsol+\zsol)} < r ,
\end{equation}
holds for $k = \kd$, then we can conclude that $\wkdd\in\Msol$. 
This is the main idea behind the model consistency of iterative regularization. For the rest of the section, we shall elaborate the details of the proof.

\subsection{Model consistency of iterative regularization}

As discussed above, the key for showing model consistency of iterative regularization is to bound the distance in \eqref{eq:distance}. 
Directly handling the distance in \eqref{eq:distance} is rather difficult, mostly due to the subgradient $\zkd$ (or the dual variable $\vkd$). 
Following the idea of \cite{garrigos2018iterative,villa2023implicit},  we rely on an intermediate point. For $\delta=0$, denote the iterates of DGD Algorithm \ref{alg:DGD} as $\zkzero, \wkzero$ and $\vkzero$, then
\begin{equation}\label{eq:decomposition}
    \begin{aligned}
        \norm{(\wkdd+\zkdd)-(\wsol+\zsol)} 
        &\leq \norm{\wkdd-\wsol} + \norm{\zkdd-\zkzero} + \norm{\zkzero - \zsol}  .
    \end{aligned}
\end{equation}
The first term on the right-hand side {\tt rhs} can be directly addressed using Lemma \ref{lem:DGD}, while the remaining two terms require more technical treatment, which is deferred to the proof.

As long as we bound the {\tt rhs} of \eqref{eq:decomposition}, we obtain the model consistency of iterative regularization with dual gradient descent.

\begin{theorem}[Model consistency of DGD]\label{thm:mc_DGD}
Consider the problems \eqref{Pdelta0} and \eqref{Ddelta0}. Suppose the following assumptions hold
\begin{itemize}
    \item[i).] $\wsol$ is the unique solution of \eqref{Pdelta0} for $\delta=0$, and $\vkzero \to \vsol$ with $\vsol$ being a dual solution.
    \item[ii).] The nondegeneracy source condition \eqref{NSC} holds for the pair $\wsol$ and $\vsol$.
    \item[iii).] $R$ is partly smooth at $\wsol$ relative to $\Msol$.
    \item[iv).] Given $\epsilon>0$, define $\mathcal{U}_\epsilon$ as in \eqref{eq:Ue} and choose a strong convexity constant $\alpha$ such that 
    \begin{equation}
        \alpha\cd\|X\|^{-1} <  \dist\Pa{ \wsol+\zsol,\bdy(\calU_{\epsilon}) },
    \end{equation}
    where $\cd$ refers to the constant in Lemma~\ref{lem:DGD}.
\end{itemize}
Let $\kd$ be the optimal early stopping time of DGD, and denote the corresponding point as $\wkdd$, then when $\delta>0$ is small enough, there holds
\[
\wkdd \in \Msol .
\]
Moreover, there exists $0<\Km<\kd<\KM$ such that $\wkd\in\Msol$ holds for all $k\in[\Km,~\KM]$. 
\end{theorem}

\begin{proof}
Let $r=\dist\Pa{ \wsol+\zsol,\bdy(\calU_{\epsilon}) }$. From the updates of DGD, we have 
$$
\wkd=\prox_{\alpha^{-1}R}(-\alpha^{-1}X^{\top} \vkd)=\nabla R^*_\alpha(-X^\top v_\delta^{(k)}),
$$ 
which by Fenchel-Young equality (see e.g. \cite[Theorem 16.29]{bauschke2017convex}) yields
$$
\zdk
:= -X^{\top} \vkd \in \partial \Ra(\wkd) . 
$$
Recall that $\zkzero = -X^\top \vkzero$, then from \eqref{eq:decomposition}
\begin{equation}\label{eq:decomposition-2}
    \begin{aligned}
        \norm{(\wkd+\zkd)-(\wsol+\zsol)} 
        &\leq \norm{\wkd-\wsol} + \norm{\zkd-\zkzero} + \norm{\zkzero - \zsol}  \\
        &\leq \norm{\wkd-\wsol} + \norm{X}\norm{\vkd-\vkzero} + \norm{X} \norm{\vkzero - \vsol}  .
    \end{aligned}
\end{equation}
In the following, we bound the three terms appearing in the {\tt rhs}.
\begin{itemize}
\item Let $c>0$. Choose  $k_\delta\in \{\lfloor c\delta^{-1}\rfloor,\ldots,2\lfloor c\delta^{-1}\rfloor\}$, and let $c_\delta$ be such that  $k_\delta=c_{\delta}\delta^{-1}$.  From Lemma \ref{lem:DGD}, we directly have 
$$
\norm{\wkdd -\wsol}
\leq \Pa{ a(c^{1/2} + 1) + bc^{-1/2} }\delta^{1/2} . 
$$
\item  Since $\Ra ^*\circ(-X^{\top})$ is convex differentiable with Lipschitz continuous gradient, we have from \cite{bauschke2017convex} that the operator 
\[
\calF = \Id + \gamma X \circ \nabla \Ra ^*\circ(-X^{\top}) 
\] 
is nonexpansive (i.e. $1$-Lipschitz), and $\calF(\vkd) = \vkd + \gamma X \wkd$. 
As a result, we get
\begin{equation}\label{eq:vkd-vkzero}
\begin{aligned}
\norm{\vkd-\vkzero}  
&= \norm{ (\vkmd+\gamma(X\wkmd-\yd)) - (\vkmzero+\gamma(X\wkmzero-\yzero)) }  \\
&\leq \norm{ (\vkmd+\gamma X\wkmd) - (\vkmzero+\gamma X\wkmzero) } + \gamma \norm{\yd-\yzero}  \\
&= \norm{ \calF(\vkmd) - \calF( \vkmzero) } + \gamma \norm{\yd-\yzero}  \\
&\leq \norm{ \vkmd - \vkmzero } + \gamma \delta . 
\end{aligned}
\end{equation}

Recalling that $v_0^{(0)}=v_\delta^{(0)=0} $Tracing back to $k=0$ yields
\[
\norm{\vkd-\vkzero}
\leq \gamma k \delta .
\]
Letting $k=\kd$ leads to
\begin{equation*}
\|\vkdd-\vkdzero\|\leq\gamma\kd\delta=\alpha\|X\|^{-2}\kd\delta=\alpha \cd\|X\|^{-2}. 
\end{equation*}
\item Lastly for $\|\vkzero-\vsol\|$, as we suppose that $\vkzero\to\vsol$, then $\norm{\vkzero-\vsol}\to 0$. Now for $k=\kd$, since $\kd$ increases as $\delta$ decreases, we have $\norm{\vkdzero-\vsol}\to 0$ as $\delta\to 0$.  
\end{itemize}
Summing the previous inequalities, we arrive at 
\begin{equation}\label{eq:decomposition-3}
    \begin{aligned}
        &\norm{(\wkdd+\zkdd)-(\wsol+\zsol)} \\
        &\leq \norm{\wkdd-\wsol} + \norm{X}\norm{\vkdd-\vkdzero} + \norm{X} \norm{\vkdzero - \vsol}  \\
        &\leq \Pa{ a(c^{1/2} + 1) + bc^{-1/2} } \delta^{1/2} + \alpha\cd \norm{X}^{-1} + \norm{X} \norm{\vkdzero-\vsol}
    \end{aligned}
\end{equation}
Since $\Pa{ a(c^{1/2} + 1) + bc^{-1/2} } \delta^{1/2} \to 0$ and $\norm{X} \norm{\vkdzero-\vsol} \to 0$ as $\delta \to 0$, if $\delta$ is such that
\[
    \alpha\cd\norm{X}^{-1} <  r,
\]
 then for $\delta$ small enough, it holds 
\begin{equation}\label{eq:model}
    \norm{(\wkdd+\zkdd)-(\wsol+\zsol)} < r
\end{equation}
which in turn implies $\wkdd\in\Msol$, and we obtain model consistency.  

Next we show that there exists an interval $[\Km,~\KM]$ with $\kd\in\left]\Km,~\KM\right[$ such that $\wkd\in\Msol$ holds for all $k\in[\Km,~\KM]$. Consider 
\[
\begin{aligned}
    &\norm{(\wkd+\zkd)-(\wsol+\zsol)}  \\
    &\leq \norm{(\wkd+\zkd)-(\wkdd+\zkdd)} + \norm{(\wkdd+\zkdd)-(\wsol+\zsol)}
\end{aligned}
\]
For the first term, we have
\[
\begin{aligned}
    &\norm{(\wkd+\zkd)-(\wkdd+\zkdd)}  \\
    &\leq \norm{\wkd-\wkdd} + \norm{\zkd-\zkdd}  \\
    &= \norm{\prox_{\alpha^{-1}R}\pa{ \alpha^{-1} \zkd }-\prox_{\alpha^{-1}R}\pa{ \alpha^{-1} \zkdd }} + \norm{\zkd-\zkdd}  \\
    &\leq \Pa{1+\sfrac{1}{\alpha}}  \norm{\zkd-\zkdd}  \\
    &\leq \Pa{1+\sfrac{1}{\alpha}} \norm{X} \norm{\vkd-\vkdd} . 
\end{aligned}
\]
From Lemma \ref{lem:convergence-vk}, when $k = \kd -1$, we have for some constant $C>0$, there holds
\[
\norm{\vkd-\vkdd} \leq \qfrac{C}{k^{1/2}}  .
\]
As a result, for a small enough $\delta$ (such that $\kd$ is large enough) and \eqref{eq:model}, inequality 
\[
    \Pa{1+\sfrac{1}{\alpha}} \norm{X} \qfrac{C}{k^{1/2}} + \norm{(\wkdd+\zkdd)-(\wsol+\zsol)} < r
\]
also holds, we have $\wkd\in\Msol$ for $k=\kd-1$. Similarly we can show that $\wkd\in\Msol$ for $k=\kd+1$. We can continue with $k=\kd\pm 2,3,...$ until the inequality fails. As a result, there exists $\Km \leq \kd \leq \KM$ such that $\wkd\in\Msol$ holds for all $k\in [\Km,~\KM ]$, and the proof is concluded.
\end{proof}

\begin{remark}
Theorem~\ref{thm:mc_DGD} establishes that for iterative regularization a result similar to that of~\cite{vaiter2017model} for variational regularization, in the sense that the model consistency holds for a sufficiently small noise level $\delta$.
However, the current proof does not allow to derive an  upper bound $\bar{\delta}$ such that model consistency happens for all $\delta < \bar{\delta}$. The reason is that we are not able to derive a convergence rate in terms of $\delta$ for the term $\norm{\vkdzero-\vsol}$ in \eqref{eq:decomposition-3}. 
\end{remark}

The following corollary shows that model consistency holds for the ADGD algorithm as well.

\begin{corollary}[Model consistency of ADGD]\label{thm:mc_ADGD}
Consider the problems \eqref{Pdelta0} and \eqref{Ddelta0}. Suppose the assumptions of Theorem \ref{thm:mc_DGD} hold. 
Let $\kd$ be the optimal early stopping time of ADGD, and denote the corresponding point as $\wkdd$, then if $\delta>0$ is sufficiently  small, there exists $\theta_\delta$ such that, if $\theta>\theta_\delta$, there holds
\[
\wkdd \in \Msol .
\]
Moreover, there exists $0<\Km<\kd<\KM$ such that $\wkd\in\Msol$ holds for all $k\in\left]\Km,~\KM\right[$. 
\end{corollary}

Below we provide a brief proof mainly highlighting the difference compared to the proof of Theorem \ref{thm:mc_DGD}, the main difficult part in this case is that the sequence generated by ADGD is not monotone. 

\begin{proof}
From the updates of ADGD, denote $\zdk := -X^{\top} \ukd$, then
$$
\zdk = \partial \Ra(\wkd) . 
$$
Again from \eqref{eq:decomposition} we have
\begin{equation}\label{eq:decomposition-2-adgd}
    \begin{aligned}
        &\norm{(\wkd+\zkd)-(\wsol+\zsol)}  \\
        &\leq \norm{\wkd-\wsol} + \norm{\zkd - \zsol}  \\
        &= \norm{\wkd-\wsol} + \norm{X} \norm{\ukd - \vsol}  \\
        &\leq \norm{\wkd-\wsol} + \norm{X}\norm{\ukd-\ukzero} + \norm{X} \norm{\ukzero - \vsol}  . 
    \end{aligned}
\end{equation}
Now we need to discuss the last line of \eqref{eq:decomposition-2-adgd} by fixing $k = \kd = \lceil c\delta^{-1/2}\rceil$. The strategy is similar to that of the proof of Theorem \ref{thm:mc_DGD}:
\begin{itemize}
    \item For the {\it first} term, owing to Lemma \ref{lem:ADGD}, there holds
$$
\| \wkdd-\wsol\|\leq \Pa{ a(c+1) + bc^{-1}} \delta^{1/2}.
$$

\item For the {\it last} term, due to the convergence we have $\norm{\ukzero - \vsol} \to 0$. 

\item 
Now for the $\norm{\ukd-\ukzero}$ of the {\it second} term. Follow the derivation of \eqref{eq:vkd-vkzero}, we have
\[
\begin{aligned}
&\norm{\ukd-\ukzero}  \\
&= \norm{ (\vkd+\gamma(X\rkd-\yd)) - (\vkzero+\gamma(X\rkzero-\yzero)) }  \\
&\leq \norm{ \calF(\vkd) - \calF( \vkzero) } + \gamma \norm{\yd-\yzero}  \\
&\leq \norm{ \vkd - \vkzero } + \gamma \delta \\
&= \norm{ \Pa{ \ukmd+\sfrac{k-2}{k-1+\theta}(\ukmd-\ukmmd) } - \Pa{ \ukmzero+\sfrac{k-1}{k+\theta}(\ukmzero-\ukmmzero) } }  + \gamma \delta  \\
&\leq \norm{\ukmd-\ukmzero} + \sfrac{k-2}{k-1+\theta} \Pa{ \norm{\ukmd-\ukmmd} + \norm{\ukmzero-\ukmmzero} }  + \gamma \delta  \\
&\leq \sum_{i=1}^{k} \sfrac{i-2}{i-1+\theta} \Pa{ \norm{\uimd-\uimmd} + \norm{\uimzero-\uimmzero} }  + \gamma k \delta  .
\end{aligned}
\]
Denote 
\begin{align}
\nonumber \beta_k &= \sum_{i=1}^{k} \sfrac{i-2}{i-1+\theta} \Pa{ \norm{\uimd-\uimmd} + \norm{\uimzero-\uimmzero} }\\
\label{eq:beta}&\leq \sfrac{k-2}{k-1+\theta} \sum_{i=1}^{k}  \Pa{ \norm{\uimd-\uimmd} + \norm{\uimzero-\uimmzero} }   . 
\end{align}
With given $\delta > 0$, $\kd = \lceil c\delta^{-1/2}\rceil$ is finite, meaning that $\beta_{\kd} \in\mathbb{R}$. Thanks to  \eqref{eq:beta}, we can choose $\theta_\delta$ in such a way that  $\beta_{\kd}$ is as small as needed. 
\end{itemize}
The above discussion implies that when $\delta$ is sufficiently small and $\theta \geq \theta_\delta$, \eqref{eq:decomposition-2-adgd} implies
\[
\norm{\wkdd-\wsol} + \norm{X}\norm{\ukdd-u_0^{(k_\delta)}} + \norm{X} \norm{u_0^{(k_\delta)} - \vsol} < r ,
\]
meaning that \eqref{eq:distance} is satsifed, and thus we obtain the model consistency of ADGD.
\end{proof}

\begin{remark}
Although model consistency for ADGD holds, the result is somewhat counterintuitive, in the sense that:
\begin{itemize}
\item When $\delta$ is very small, $\kd$ becomes very large, which in turn makes $\beta_{\kd}$ large.
    \item To bound $\norm{(\wkdd+\zkdd)-(\wsol+\zsol)}$, an even larger value of $\theta$ is required. This causes ADGD to behave more like DGD, despite still converging at an 
$(1/k^2)$ rate.
\end{itemize}
The main reason, as emphasized above, is that the sequence generated by ADGD is not monotonic; the inertial term introduces a dependence on both the current point $\ukd$ and its history $\ukmd$.
In practice, however, model consistency of ADGD is often observed also for small values of $\theta$; see Section \ref{sec:ne} for more details.
\end{remark}


\subsection{Model consistency of continuous time dynamics} 

In \cite{calatroni2021accelerated}, the authors studied the continuous-time dynamics of DGD and its accelerated variant, demonstrating their iterative regularization properties. Our model consistency results can also be extended to the continuous-time setting.
For simplicity, we focus here only on the continuous-time dynamics of DGD.
The DGD iteration can be reformulated as:
\[
\begin{aligned}
    \wkd &= \nabla \Ra ^*(-X^{\top}\vkd) ,  \\ 
    \qfrac{\vkpd - \vkd}{\gamma} &= X\wkd - \yd = - \nabla \phi(\vkd)  .
\end{aligned}
\]
Let $\gamma>0$ be a time step and let $t_k = k\gamma$. The iterates  $\vkd$ and $\wkd$ can then be interpreted as a discretization  of the continuous time dynamics:
\begin{equation}\label{ODE:general}
    v_{\delta}(0)=\vzero=0,~~\begin{cases}
        \wdt=\nabla \Ra^*(-X^{\top}\vdt)\\
        \dvdt = - \nabla \phi(\vdt) .
     \end{cases}
 \end{equation}

\begin{corollary}[Model consistency of continuous time DGD]\label{thm:mc_ode}
Consider the problems \eqref{Pdelta0} and \eqref{Ddelta0}. Suppose the assumptions of Theorem \ref{thm:mc_DGD} hold. 
Let $t_{\delta}$ be the optimal early stopping time of \eqref{ODE:general}, and denote the corresponding point as $\wdtd$, then when $\delta>0$ is small enough, there holds
\[
   \wdtd \in \Msol .
\]
Moreover, there exists $0<\Tm\leq\td\leq\TM$ such that $\wdt\in\Msol$ holds for all $t\in[\Tm,~\TM]$.

\end{corollary}

Figure \ref{fig:ode} shows a numerical comparison between DGD and its continuous-time counterpart.

\section{Local characterization of the regularization error}\label{sec:llc}

From the previous section, we have seen that the model consistency property holds for a range of points. 

Since along $\Mwsol$, function $R$ is $C^2$-smooth near $\wsol$, this implies that in the interval $\Kd$, we can obtain a finer characterization of the error $\norm{\wkd-\wsol}$, which we call as {\it regularization error}. 
This result is motivated by the local linear convergence of first-order method, see for instance \cite{liang2014local,liang2017activity,poon2019trajectory,molinari2019convergence,klopfenstein2024local} and the references therein. 

For the sake of simplicity, in this section we only focus on the case where $\Msol$ is affine, this includes $\ell_1$-norm, TV semi-norm and $\ell_{1,2}$-norm in Table \ref{table:examples-psf}. 
Since $\Msol$ is affine, it can be represented as 
\[
\Msol = \wsol + \Twsol ,
\]
where $T_{\wsol}$ is the tangent space of $\Msol$ at $\wsol$. Moreover, for any $w\in\Msol$, there holds
$$
w-\wsol = P_{\Twsol}(w-\wsol) ,
$$
where $P_{\Twsol}$ is the projection operator associated with $\Twsol$. 
Take point $\wKmd$, and define the following matrix
\[
\MIR := \PTwsol - \qfrac{1}{\|X\|^2} \PTwsol \Pa{ \Id + \nabla^2_{\Mwsol} R(\wKmd) }^{-1} \XTwsol^{\top}\XTwsol ,
\]
where $X_{\Twsol} := X \PTwsol$.

The spectral radius of $\MIR$ is denoted by $\rhoMIR$, and to characterize its property we need the following  assumption between $\Twsol$ and $X$.

\begin{assumption}[Restricted injectivity]
Suppose the following condition holds
\begin{equation}\tag{$\INJ_{\Twsol}$}\label{INJ}
\Ker\pa{ X } \cap \Twsol = \{0\} ,
\end{equation}
where $\Ker(X)$ is the null space of $X$. 
\end{assumption}

\begin{remark}
Condition \eqref{INJ} is also considered in \cite{liang2017activity} where the authors study the local linear convergence of the proximal gradient type methods, it implies $ X_{\Twsol}^\top  X_{\Twsol}$ is positive definite along $\Twsol$. 
\end{remark}

To characterize the regularization error $\norm{\wkd-\wsol}$, we suppose that the early stopping point $\wkdd$ has the minimal distance to $\wsol$, and denote it by $d^\dagger_{\delta}$, that is
\[
d^\dagger_{\delta} := \norm{\wkdd-\wsol} = \min_{k\in\Kd} \norm{\wkd-\wsol} . 
\]
We have the following result.

\begin{theorem}[Characterization of regularization error]\label{thm:reg-error}
Let $R$ be partly smooth at $\wsol$ relative to $\Msol$ and suppose that $\Msol = \wsol + \Twsol$ is affine. Let the assumptions in Theorem~\ref{thm:mc_DGD} and condition \eqref{INJ} hold. Then there holds
\begin{itemize}
    \item[i).] The spectral radius satisfies $\rhoMIR < 1$. 

    \item[ii).] Fixed stepsize as $\gamma = \alpha/\norm{X}^2$, then
\begin{equation}\label{eq:linarization}
{\wkpd - \wkd} 
= \MIR \pa{\wkd - \wkmd} + \eta^{(k)}
\end{equation}
where $\eta^{(k)} = o(\norm{\wkpd-\wKmd})  + o(\norm{\wkd-\wKmd})  + o(\norm{\wkmd-\wKmd})$. 
    \item[iii).] If $\Kd$ contains more than three points, then $\forall k\in [\Km, \KM[$ and any $\rho\in ]\rhoMIR, 1[$
\[
    \|\wkpd-\wkd\| 
    \leq O( \rho^{k-\Km} ) . 
\]
\item[iv).] With the above $\rho$, there exists a constant $D>0$ such that the regularization error~satisfies
\[
\norm{\wkd-\wsol} \leq d^\dagger_{\delta} + D \rho^{\min\{k, \kd\}-\Km} \times   \sfrac{ 1 - \rho^{|k-\kd|} }{1 - \rho}  . 
\]
\end{itemize}

\end{theorem}

\begin{proof}  In the following we prove the claims step by step.
\paragraph{i). Spectral radius:}
Since $\Mwsol$ is affine, $\nabla^2_{\Mwsol} R(\wKmd)$ is symmetric positive semi-definite \cite[Lemma 4.2]{liang2017activity}. Then $\Id + \sfrac{1}{\alpha}  \nabla^2 \widetilde{R}(\wKmd)$ is invertible, and denote
\[
W_{\Twsol} 
= \PTwsol + \sfrac{1}{\alpha}  \nabla^2_{\Mwsol} R(\wKmd)
=  \Pa{ \Id + \sfrac{1}{\alpha}  \nabla^2_{\Mwsol} R(\wKmd) } \PTwsol ,
\]
whose left inverse reads
\[
W_{\Twsol}^{-1}
= \PTwsol \Pa{ \Id + \sfrac{1}{\alpha}  \nabla^2_{\Mwsol} R(\wKmd) }^{-1}   ,
\]
as we have $W_{\Twsol}^{-1} W_{\Twsol} = \PTwsol$ which is the identity mapping on $\Twsol$. 
Clearly, all the non-zero eigenvalues of  $W_{\Twsol}^{-1}$ are in $]0,1]$. 
Note that $\Pa{ \Id + \sfrac{1}{\alpha}  \nabla^2_{\Mwsol} R(\wKmd) }^{-1} \XTwsol^{\top}\XTwsol$ is similar to a symmetric matrix, hence all its eigenvalues are real. As a result, the eigenvalues of $$\PTwsol \Pa{ \Id + \sfrac{1}{\alpha}  \nabla^2_{\Mwsol} R(\wKmd) }^{-1} \XTwsol^{\top}\XTwsol = W_{\Twsol}^{-1} \XTwsol^{\top}\XTwsol $$ are also real. Let  $d = {\dim}(\Twsol)$ be the dimension of $\Twsol$ and $\sigma_i> 0, i=1,...,d$ be the non-zero eigenvalues of $W_{\Twsol}^{-1} \XTwsol^{\top}\XTwsol$, the corresponding eigenvalues of $\MIR$ then read
\[
\xi_i = 1 - \sfrac{1}{\norm{X}^2}\sigma_i,~~ i=1,...,d .
\]
Since $\norm{\XTwsol} \leq \norm{X}$ and $\norm{W_{\Twsol}^{-1}} \leq 1$, we have $\sfrac{1}{\norm{X}^2}\sigma_i \in ]0, 1]$, which in turn means $\xi_i < 1$, and we conclude the first claim. 

\paragraph{ii). Linearization of iteration:}
From the update of DGD iteration, for any $k\in [\Km, \KM [$, we get
\[
\begin{aligned}
-\sfrac{1}{\alpha} X^{\top} \vkd - \wkd 
    &\in \alpha^{-1}\partial R(\wkd)   .
\end{aligned}
\]
Since $\Mwsol$ is affine, projecting the inclusion onto $\Twsol$ yields the Riemannian gradient
\begin{equation}\label{eq:wkd-lin}
\begin{aligned}
&\PTwsol \Pa{ -\sfrac{1}{\alpha} X^{\top} \vkd - \wkd }  \\
&= \alpha^{-1} \nabla_{\Mwsol} R(\wkd)   \\
&= \alpha^{-1} \nabla_{\Mwsol} R(\wKmd) + \alpha^{-1} \nabla_{\Mwsol}^2 R(\wKmd) (\wkd - \wKmd) + o(\norm{\wkd-\wKmd})   \\
\end{aligned}
\end{equation}
where the second equality comes from the Taylor expansion with respect to $\wKmd$; see Lemma~\ref{lem:taylor}. Repeat the above procedure to $\wkpd$, we get
\begin{equation}\label{eq:wkpd-lin}
\begin{aligned}
&\PTwsol \Pa{ -\sfrac{1}{\alpha} X^{\top} \vkpd - \wkpd }  \\
&= \alpha^{-1} \nabla_{\Mwsol} R(\wKmd) + \alpha^{-1} \nabla_{\Mwsol}^2 R(\wKmd) (\wkpd - \wKmd) + o(\norm{\wkpd-\wKmd}) .
\end{aligned}
\end{equation}
Taking the difference of \eqref{eq:wkd-lin} and \eqref{eq:wkpd-lin}, we get
\[
\begin{aligned}
&\alpha^{-1} \nabla_{\Mwsol}^2 R(\wKmd) (\wkpd - \wkd) + o(\norm{\wkpd-\wKmd}) + o(\norm{\wkd-\wKmd}) \\
&= -\sfrac{1}{\alpha} \PTwsol X^{\top}\pa{\vkpd - \vkd} - \PTwsol (\wkpd - \wkd)  \\
&= -\sfrac{\gamma}{\alpha} \PTwsol X^{\top} \pa{X \wkd - \yd} - \PTwsol (\wkpd - \wkd) .
\end{aligned}
\]
Recall that $W_{\Twsol} = \PTwsol + \sfrac{1}{\alpha}  \nabla^2_{\Mwsol} R(\wKmd)$, this means from the above derivation we obtain
\begin{equation}\label{eq:wkpd-lin2}
\begin{aligned}
&W_{\Twsol} (\wkpd - \wkd) + o(\norm{\wkpd-\wKmd}) + o(\norm{\wkd-\wKmd})  \\
&= -\sfrac{\gamma}{\alpha} \PTwsol X^{\top} \pa{X \wkd - \yd} .
\end{aligned}
\end{equation}
Similarly, we have
\begin{equation}\label{eq:wkd-lin2}
\begin{aligned}
&W_{\Twsol} \pa{\wkd - \wkmd}  + o(\norm{\wkd-\wKmd})  + o(\norm{\wkmd-\wKmd})  \\
&= -\sfrac{\gamma}{\alpha} \PTwsol X^{\top} \pa{X \wkmd - \yd} .
\end{aligned}
\end{equation}
Combining \eqref{eq:wkpd-lin2} and \eqref{eq:wkd-lin2} leads to
\[
\begin{aligned}
W_{\Twsol}\pa{\wkpd - \wkd}  
&=  W_{\Twsol} \pa{\wkd - \wkmd} -\qfrac{\gamma}{\alpha} \PTwsol X^{\top} X \pa{\wkd - \wkmd} + \eta^{(k)}  \\
&= \Pa{W_{\Twsol} -\qfrac{\gamma}{\alpha} \PTwsol X^{\top} X \PTwsol } \pa{\wkd - \wkmd} + \eta^{(k)}  \\
&= \Pa{W_{\Twsol} -\qfrac{\gamma}{\alpha} \XTwsol^{\top} \XTwsol } \pa{\wkd - \wkmd} + \eta^{(k)} ,
\end{aligned}
\]
where $\eta^{(k)} = o(\norm{\wkpd-\wKmd})  + o(\norm{\wkd-\wKmd})  + o(\norm{\wkmd-\wKmd}) $. 
Let $\gamma = \alpha/\norm{X}^2$, taking the left inverse of $W_{\Twsol}$ we get
\[
\begin{aligned}
{\wkpd - \wkd} 
&= W_{\Twsol}^{-1} \bPa{ W_{\Twsol} - \sfrac{1}{\norm{X}^2} \XTwsol^{\top} \XTwsol } \pa{\wkd - \wkmd} + W_{\Twsol}^{-1} \eta^{(k)} \\
&= \MIR \pa{\wkd - \wkmd} + \eta^{(k)} ,
\end{aligned}
\]
we abuse the notation that $\eta^{(k)} = W_{\Twsol}^{-1} \eta^{(k)}$. 
To this point, the second claim is proved.

\paragraph{iii). Linear decrease of $\norm{\wkpd-\wkd}$:}
Tracking back to $\Km$, we get
\[
\begin{aligned}
{\wkpd - \wkd}
&= \MIR \pa{\wkd - \wkmd} + \eta^{(k)}   \\
&= \MIR^{k-\Km} \pa{ w_{\delta}^{(\Km+1)}  - w_{\delta}^{(\Km)} } + \sum_{i=\Km + 1}^{k} \MIR^{k-i} {\eta}^{(i)}  , 
\end{aligned}
\]
from which we get 
\[
\norm{\wkpd - \wkd}
\leq \norm{ \MIR^{k-\Km} } \norm{\ w_{\delta}^{(\Km+1)}  - w_{\delta}^{(\Km)}} + \norm{ \msum_{i=\Km + 1}^{k} \MIR^{k-i} {\eta}^{(i)}  }  .
\]
The spectral radius theorem states that there exists $C>0$ such that 
\[
    \norm{ \MIR^{k-\Km} } \leq C \rho^{k-\Km} ,
\]
for any $\rho\in]\rhoMIR, 1[$. 
Since $k\leq\KM$ is finite, $\norm{ \msum_{i=\Km + 1}^{k} \MIR^{k-i} {\eta}^{(i)}  }$ is also finite, then
\[
    \norm{ \msum_{i=\Km + 1}^{k} \MIR^{k-i} {\eta}^{(i)}  }
    = \rho^{k-\Km} \times  \sfrac{1}{\rho^{k-\Km}} \norm{ \msum_{i=\Km+1}^{k} \MIR^{k-i} {\eta}^{(i)}  }  .
\]
Let 
\[
D = \max\bBa{ C,~ \sfrac{1}{\rho^{\KM-\Km}} \norm{ \msum_{i=\Km + 1}^{\KM} \MIR^{k-i} {\eta}^{(i)} } } ,
\]
we have 
\[
    \norm{\wkpd - \wkd} 
    \leq D \rho^{k-\Km}
\]
for $k\in [\Km, \KM[$, which is the third claim.

\paragraph{iv). Behavior of the regularization error:}
For the regularization error, we have the following decomposition
\begin{equation}\label{eq:reg-error-decom}
\begin{aligned}
    \|\wkd-\wsol\| 
    &= \|\wkd-\wkdd + \wkdd - \wsol\| \\
    &\leq \|\wkd-\wkdd \| + \|\wkdd - \wsol\| \\
    &= \|\wkd-\wkdd \| + d^\dagger_{\delta} . 
    \end{aligned}
\end{equation}
We need to take care of the first, for which we have two cases depending on the relation between $k$ and $\kd$.
\begin{itemize}
    \item {\bf Case $k\leq\kd$:} For this case, we have
    \[
\begin{aligned}
 \|\wkd-\wkdd \| 
 &\leq \sum_{i=0}^{ \kd-k-1} \norm{ w_{\delta}^{(k+1+i)} - w_{\delta}^{(k+i)} } \\ 
 &\leq \sum_{i=0}^{\kd-k-1} D \rho^{k+i-\Km} 
 = D \rho^{k-\Km} \times   \sfrac{ 1 - \rho^{\kd-k} }{1 - \rho}  .
\end{aligned}
\]
    \item {\bf Case $k>\kd$:} For this case, 
    \[
\begin{aligned}
 \|\wkd-\wkdd \| 
 &\leq \sum_{i=0}^{ k - \kd-1} \norm{ w_{\delta}^{(\kd+1+i)} - w_{\delta}^{(\kd+i)} } \\ 
 &\leq \sum_{i=0}^{k - \kd-1} D \rho^{\kd+i-\Km} 
 = D \rho^{\kd-\Km} \times   \sfrac{ 1 - \rho^{k-\kd} }{1 - \rho}  .
\end{aligned}
\]
\end{itemize}
Combining the two cases together, we get
\[
\|\wkd-\wkdd \| 
\leq D \rho^{\min\{k, \kd\}-\Km} \times   \sfrac{ 1 - \rho^{|k-\kd|} }{1 - \rho} .
\]
Plugging the above result into \eqref{eq:reg-error-decom} we obtain the desired claim, and the proof is concluded.
\end{proof}


\begin{remark}$~$
\begin{itemize}
    \item The small $o$-term $\eta^{(k)}$ in the linearization \eqref{eq:linarization} vanishes when $R$ is locally polyhedral around $\wsol$ along $\Mwsol$. 

    \item The linear decrease of the residual error $\norm{\wkpd-\wkd}$ holds only for the model consistency interval $[\Km, \KM[$. When $R$ is moreover locally polyhedral around $\wsol$ along $\Mwsol$, then $\rho = \rhoMIR$ is allowed. 

    \item The result iv) of Theorem \ref{thm:reg-error} indicates that in the interval of model consistency, the regularization error decreases to the optimal $d^\dagger_\delta$ at least linearly, then start to increase with an upper bound $D \rho^{\kd-\Km}$; see Figure \ref{fig:regerror-l1} for illustration.

    \item Since Theorem \ref{thm:reg-error} assumes that the manifold $\Msol$ is affine, it does not apply to the nuclear norm, whose associated manifold is curved. By incorporating additional tools from Riemannian geometry --- such as parallel transport \cite{liang2017activity} --- Theorem \ref{thm:reg-error} can be extended to more general manifolds $\Msol$, including the nuclear norm. However, to maintain clarity and focus, we omit this theoretical extension and instead provide a numerical illustration in Figure \ref{fig:regerror-nn} of the next section.
\end{itemize}

\end{remark}

\section{Numerical Experiments}\label{sec:ne}
 In this section, we present numerical experiments to support our theoretical findings. We first illustrate the model consistency result, and then the local behavior of the regularization error. Four low-complexity regularization functions $R$ are considered: the $\ell_1$ and $\ell_{1,2}$ norms, one-dimensional total variation (1D-TV), and the nuclear norm.

\subsection{Model consistency}

\paragraph{Problem settings}
Recall that in the forward model \eqref{eq:yd}, $X\in\bbR^{n\times p}$ is the forward operator and $\varepsilon \sim \mathcal{N}(0, \delta^2)$ is the noise.
For each choice of $R$, the corresponding problem settings are
\begin{itemize} 
    \item {\bf $\ell_1, \ell_{1,2}$-norms} $(n,p) = (100, 500)$, $\wsol \in \calW$ has $5$ non-zero elements, $X\in\bbR^{n\times p}$ with element sampled from standard normal distribution normalized by $\sqrt{n}$. For $\ell_{1,2}$-norm, the group size is $5$, without overlapping.  

    \item {\bf 1d-TV} $(n,p) = (20, 50)$, $\wsol \in \calW$, $\DDIF$ is the discrete gradient operator, and $\|\DDIF \wsol\|_1=1$. $X\in\bbR^{n\times p}$ with element sampled from standard normal distribution normalized by $\sqrt{n}$. 

    \item {\bf Nuclear norm} $\wsol\in\bbR^{20\times 20}$ is a square matrix with $\rank(\wsol)=1$, $X\in \bbR^{20\times 20}$ is a $0$-$1$ projection mask with  $50\%$ of the entries equal to $1$.  
\end{itemize}
The  noise level is determined by the signal-to-noise-ratio (SNR), which is defined by
\[
{\tt SNR} = 20 \log\Ppa{ \frac{\norm{y}}{\norm{\varepsilon}} } . 
\]
In our experiment, ${\tt SNR}=40$ is considered. 

In the model \eqref{Pdelta0}, we fix $\alpha=0.01$. 
For both DGD and ADGD methods, the step-size is $\gamma= \alpha/\norm{X}^2$ and the parameter $\theta$ for ADGD is set as $5$. Both algorithms are initialized at $0$, as specified in Algorithm~\ref{alg:DGD}  and \ref{alg:ADGD}.

\paragraph{Case $\ell_1$-norm}
We first present results for the $\ell_1$-norm case, shown in Figure~\ref{fig:lasso}: the {\it left} column corresponds to DGD, and the {\it right} column to ADGD.

In each plot, the {\it blue} line represents the normalized error $\frac{ \norm{\wkd - \wsol} }{ \norm{\wsol} }$, while the {\it red} line shows the support size of $\wkd$. The following observations can be made:

\begin{itemize}
\item {\bf Support size $|\supp(\wkd)|$:} Initializing at 0, the support of $\wkd$ grows, and for a certain range of iterations $k$, we observe that $|\supp(\wkd)| = 5$, indicating model consistency. As the iterations continue, the support size increases beyond 5.
\item {\bf Iterative regularization:} The normalized error \(\frac{ \norm{\wkd - \wsol} }{ \norm{\wsol} }\) attains its minimum within the model-consistent regime. Notably, the error remains nearly constant throughout this interval, suggesting that model consistency can serve as a practical criterion for determining the stopping time.
\end{itemize}

\begin{figure}[!htb]
\centering
 \subfloat[DGD]
  {
      \includegraphics[width=0.475\textwidth, trim={0mm 0mm 0mm 0mm},clip]{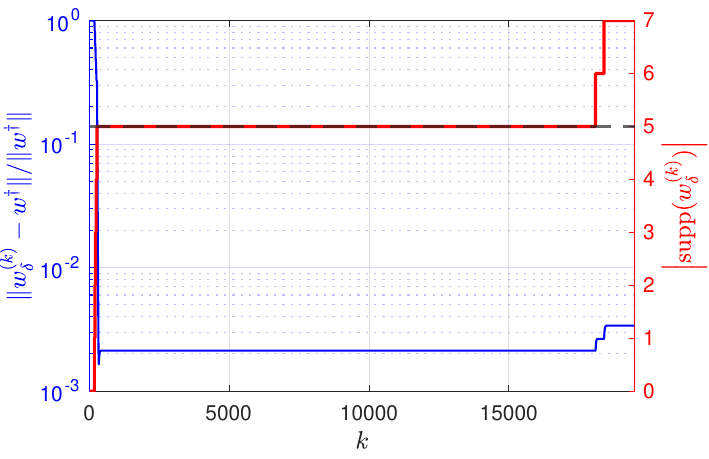}
  }
  \subfloat[ADGD]
  {
      \includegraphics[width=0.475\textwidth, trim={0mm 0mm 0mm 0mm},clip]{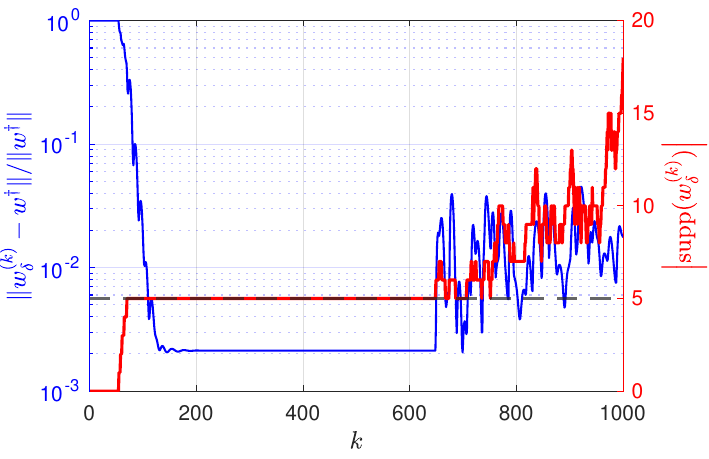}
  }
  \caption{Model consistency result of $\ell_1$-norm.}  
  \label{fig:lasso}            
\end{figure}

For this problem, we also conduct an experiment to determine the smallest {\tt SNR} value for which model consistency occurs. Specifically, we vary the {\tt SNR} in the range $[10, 60]$. For each value, we run DGD and record the support size of the iterate $\wkdd$ that is closest to $\wsol$ in terms of Euclidean distance. The support size $|\supp(\wkdd)|$ as a function of {\tt SNR} is shown in Figure~\ref{fig:differror}. From the plot, we observe that model consistency is achieved when {\tt SNR} exceeds 21.

\begin{figure}[!htb]
\centering 
\includegraphics[width=0.6\textwidth, trim={1mm 1mm 1mm 1mm},clip]{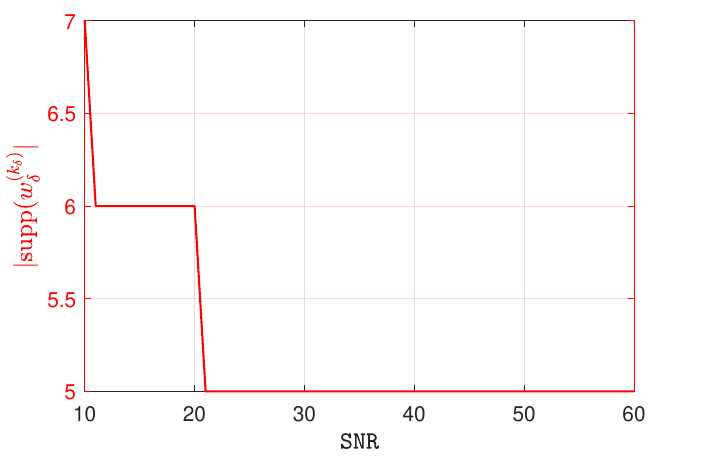}
\caption{Model consistency of DGD over different  {\tt SNR} values.}
\label{fig:differror}
\end{figure}


\paragraph{Other cases}

Figures~\ref{fig:othercases-l12}-\ref{fig:othercases-nn}  present the model consistency results for the $\ell_{1,2}$-norm (Figures~\ref{fig:othercases-l12}), 1d-TV (Figures~\ref{fig:othercases-1dtv}), and nuclear norm (Figures~\ref{fig:othercases-nn}). The observed behavior is largely consistent with that of the $\ell_1$-norm case. In addition, we note the following:

\begin{itemize}
\item For the $\ell_{1,2}$-norm, the non-zero groups of $\wkd$ within the model consistency interval exactly match those of $\wsol$.
\item The 1d-TV regularizer appears more sensitive to noise. Although model consistency holds in this particular example, there are instances—under the same {\tt SNR} value—where it fails. For further discussion on the structure sensitivity to noise, we refer the reader to \cite{vaiter2017model}.
\end{itemize}

\begin{figure}[!htb]
\centering
   \subfloat[$\ell_{1,2}$-norm: DGD]  
  {
      \includegraphics[width=0.475\textwidth, trim={0mm 0mm 0mm 0mm},clip]{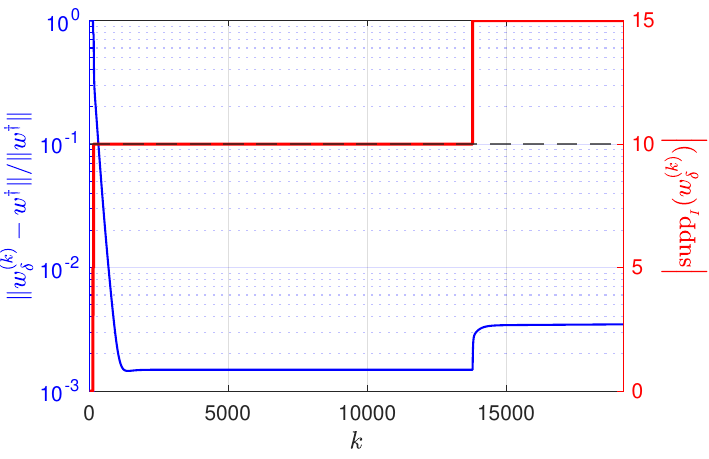}
  }
  \subfloat[$\ell_{1,2}$-norm: ADGD]
  {
      \includegraphics[width=0.475\textwidth, trim={0mm 0mm 0mm 0mm},clip]{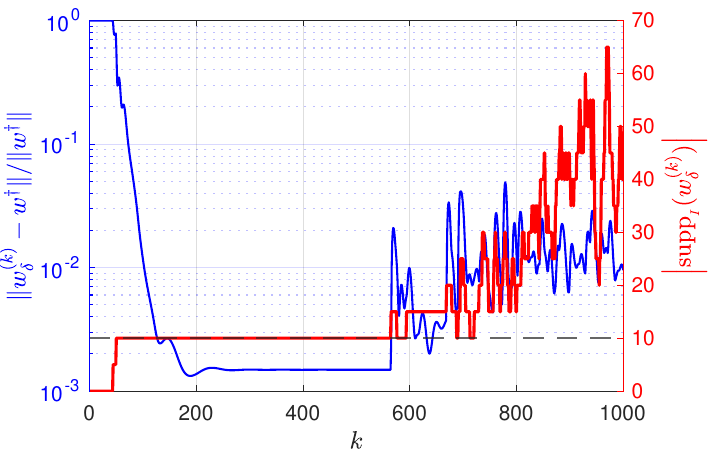}
  } 
  \caption{Model consistency result of $\ell_{1,2}$-norm.} \label{fig:othercases-l12}          
\end{figure}

\begin{figure}[!htb]
\centering
     \subfloat[1d-TV: DGD] 
  {
      \includegraphics[width=0.475\textwidth, trim={0mm 0mm 0mm 0mm},clip]{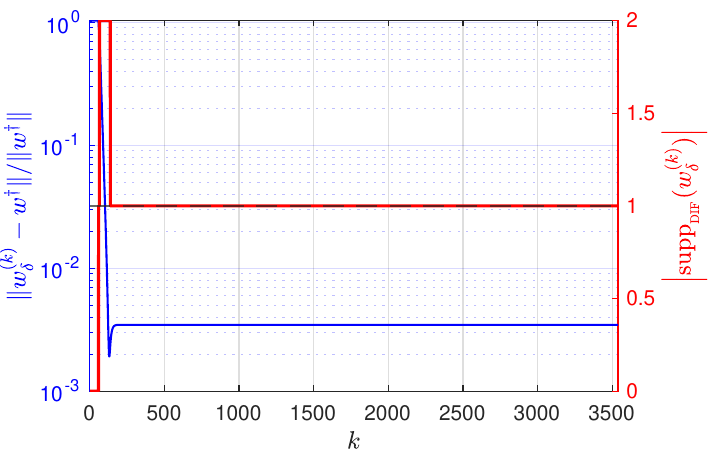}
  }
  \subfloat[1d-TV: ADGD]
  {
      \includegraphics[width=0.475\textwidth, trim={0mm 0mm 0mm 0mm},clip]{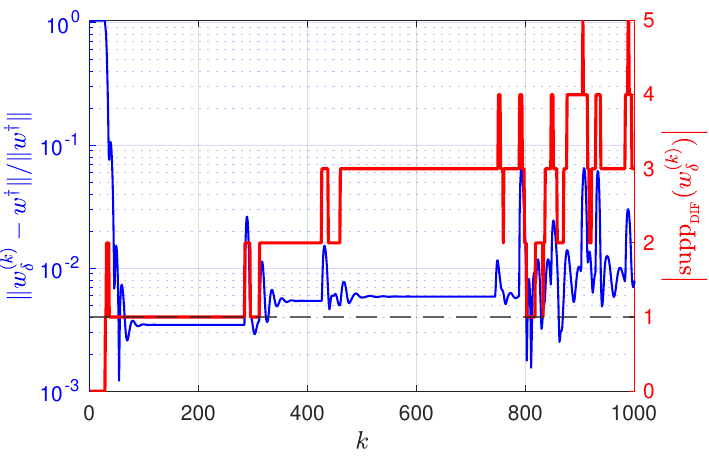}
  } 
  \caption{Model consistency result of 1d-TV.} \label{fig:othercases-1dtv}          
\end{figure}

\begin{figure}[!htb]
\centering
     \subfloat[Nuclear norm: DGD] 
  {
      \includegraphics[width=0.475\textwidth, trim={0mm 0mm 0mm 0mm},clip]{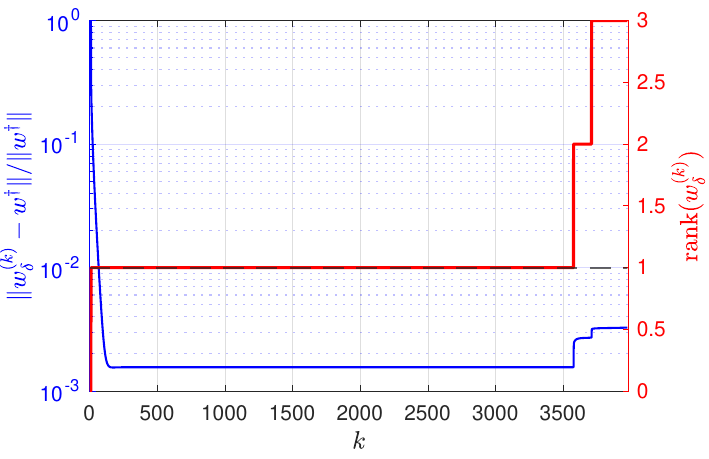}
  }
  \subfloat[Nuclear norm: ADGD]
  {
      \includegraphics[width=0.475\textwidth, trim={0mm 0mm 0mm 0mm},clip]{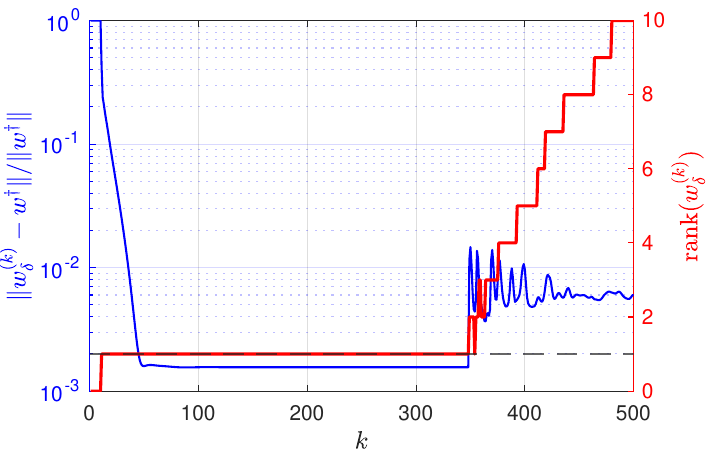}
  }
  \caption{Model consistency result of nuclear norm.} \label{fig:othercases-nn}          
\end{figure}

\paragraph{Model consistency of the continuous dynamics} 

We also verify model consistency for the continuous dynamical system \eqref{ODE:general}, using the $\ell_1$ norm with {\tt SNR} = 40. The system is solved using the fourth-order Runge–Kutta method. The results, shown in Figure~\ref{fig:ode}, are presented alongside those of the DGD method. It can be seen that the continuous dynamical system is very close to DGD, and achieves model consistency.

\begin{figure}[!htb]
\centering
\includegraphics[width=0.65\textwidth]{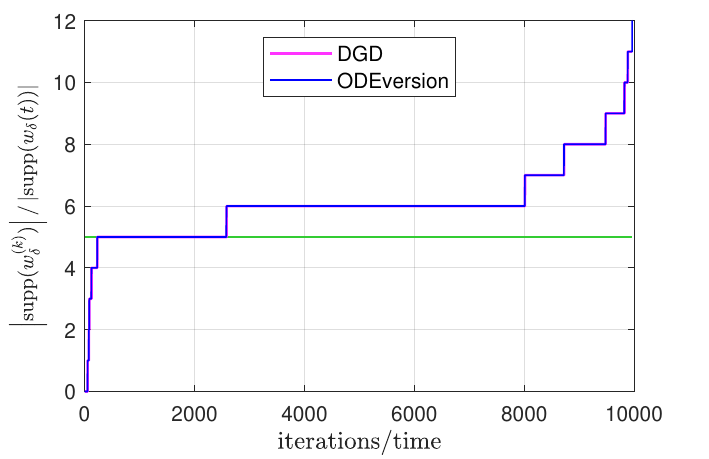}
\caption{Model consistency of the dynamical system \eqref{ODE:general}.}\label{fig:ode}
\end{figure}


\subsection{Local behavior of the regularization error}

From the above illustration, model consistency indeed occurs in an interval. In this part, we illustrate the result of Theorem \ref{thm:reg-error}, the local behavior of the regularization error.

Due to numerical precision limitations, the problem sizes considered in this section are smaller than those in the previous experiments. Specifically, we use the following settings:

\begin{itemize}
\item {\bf $\ell_1$ and $\ell_{1,2}$ norms:} We set $(n, p) = (20, 100)$, where $\wsol \in \calW$ has 2 non-zero entries. The measurement matrix $X \in \mathbb{R}^{n \times p}$ is generated with i.i.d. standard normal entries and normalized by $\sqrt{n}$. For the $\ell_{1,2}$-norm case, the group size is 2 with non-overlapping groups.

\item {\bf Nuclear norm:} \(\wsol \in \mathbb{R}^{10 \times 10}\) is a rank-one square matrix. The observation operator \(X \in \mathbb{R}^{10 \times 10}\) is a binary projection mask, where only 40\% of the entries equal to 1.

\end{itemize}

For noise $\varepsilon$, ${\tt SNR}=40$ is considered. 
In the model \eqref{Pdelta0}, same as before we fix $\alpha=0.01$. 
The step-size is $\gamma= \alpha/\norm{X}^2$ and the starting point of both algorithms is set as $0$, as indicated in the algorithms. 

In the experiment, the MATLAB variable-precision arithmetic {\tt vpa} is used for high precision. 

\paragraph{$\ell_1$-norm}
In Figure \ref{fig:regerror-l1} we provide the observations of local linear convergence for $\ell_1$-norm, the left figure shows the behavior of $\norm{\wkpd-\wkd}$, while the right figure displays the convergence of $\norm{\wkd-\wsol}$ to $\norm{\wkdd-\wsol}$ over the model consistency interval.  
Note that for both figures, around iteration steps $2,000 - 3,000$, the black lines stagnate which is due to numerical precision limit. Otherwise, the observation complies with our theory. 
For this particular example, the $\wkdd$ appears at the end of the model consistency interval. 

\begin{figure}[!htb]
\centering
 \subfloat[Residual $\norm{\wkpd-\wkd}$]  
  {
      \includegraphics[width=0.45\textwidth, trim={5mm 1mm 5mm 2mm},clip]{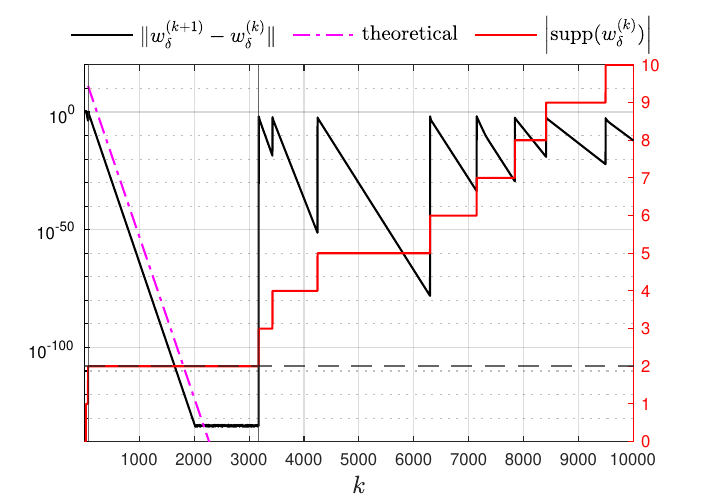}
  }
  \subfloat[Regularization error]
  {
      \includegraphics[width=0.45\textwidth, trim={5mm 1mm 5mm 2mm},clip]{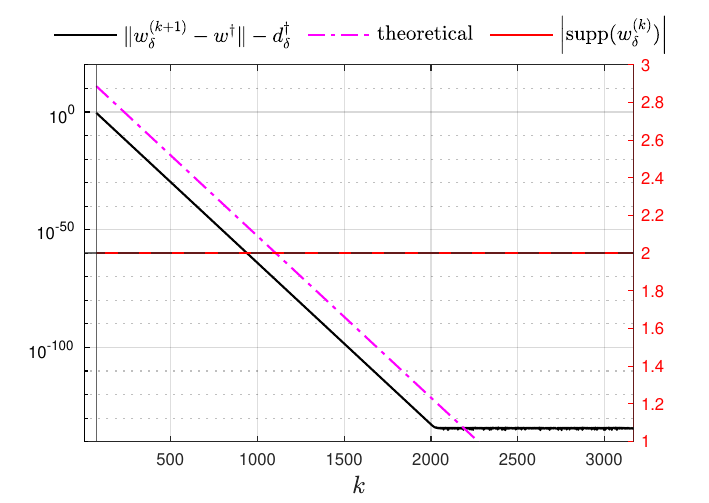}
  }
  \caption{Convergence behavior of regularization error of DGD for $\ell_1$-norm.}   
  \label{fig:regerror-l1}             
\end{figure}

In Figure \ref{fig:regerror-l12}  we provide the experiment on $\ell_{1,2}$-norm. The behavior of residual error is the same as the $\ell_1$-norm. While for the regularization error, the optimal $\wkdd$ appears inside the interval, and the error decreases first and then increases, which matches perfectly with our theory. 

\begin{figure}[!htb]
\centering
  \subfloat[Residual  $\norm{\wkpd-\wkd}$]   
  {
      \includegraphics[width=0.45\textwidth, trim={5mm 1mm 5mm 2mm},clip]{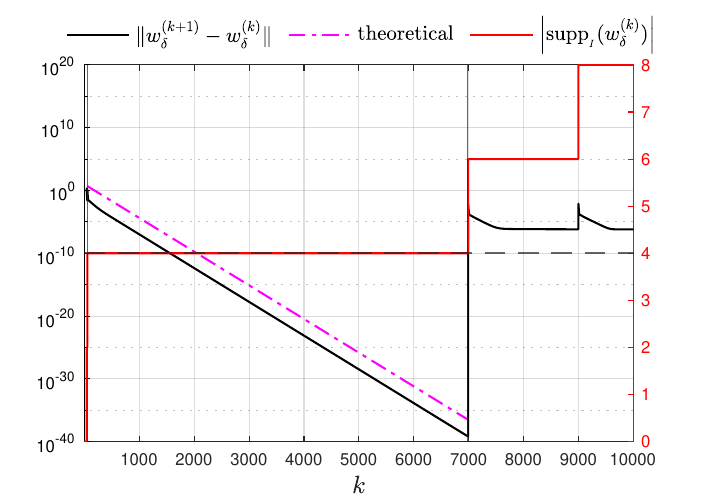}
  }
  \subfloat[Regularization error]
  {
      \includegraphics[width=0.45\textwidth, trim={5mm 1mm 5mm 2mm},clip]{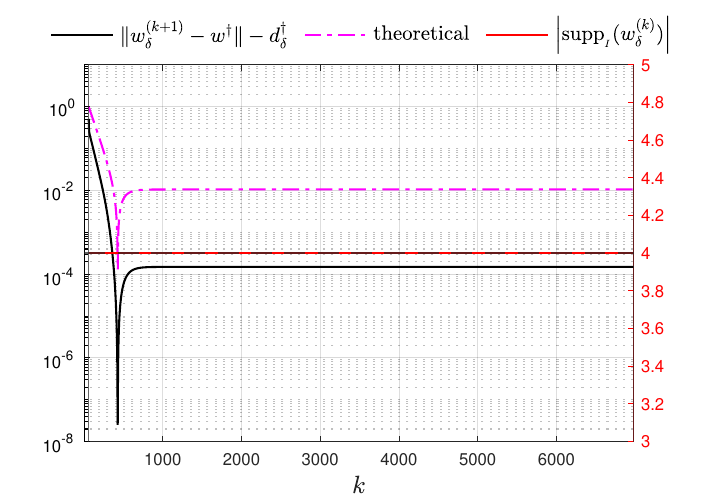}
  }
  \caption{Convergence behavior of regularization error of DGD for $\ell_{1,2}$-norm.}   
  \label{fig:regerror-l12}             
\end{figure}

Lastly we provide an example for the nuclear norm in Figure \ref{fig:regerror-nn}, thought our theory does not cover this case, similar observation is obtained. Note that the stagnation of the black line in the left figure is due to numerical precision.

\begin{figure}[!htb]
\centering
  \subfloat[Residual  $\norm{\wkpd-\wkd}$]   
  {
      \includegraphics[width=0.45\textwidth, trim={5mm 1mm 5mm 2mm},clip]{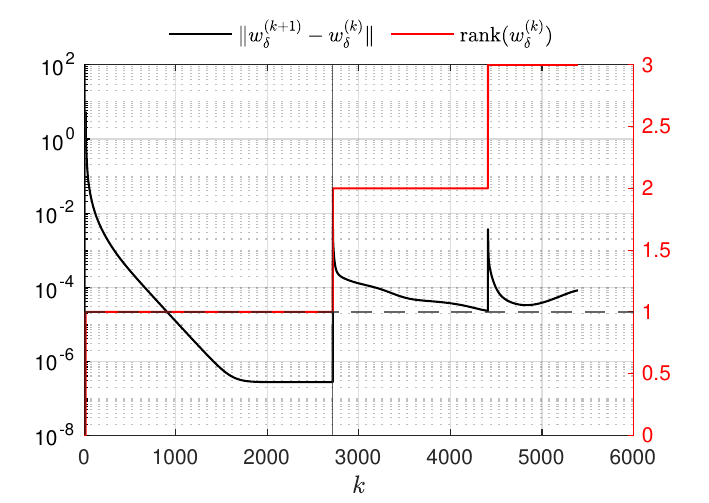}
  }
  \subfloat[Regularization error]
  {
      \includegraphics[width=0.45\textwidth, trim={5mm 1mm 5mm 2mm},clip]{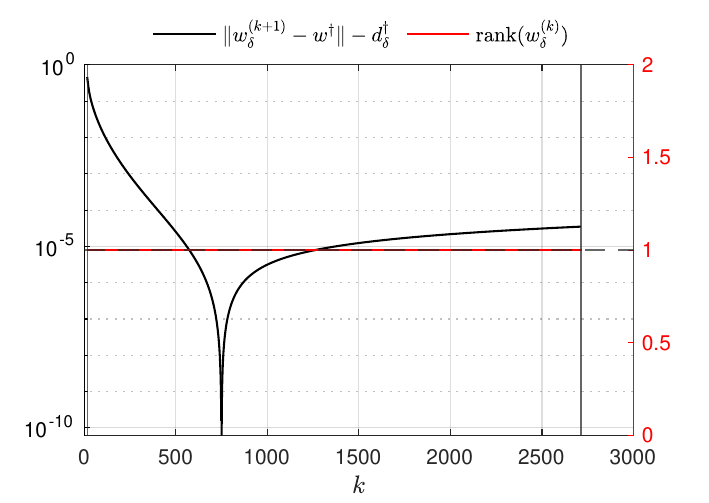}
  }
  \caption{Convergence behavior of regularization error of DGD for nuclear norm.}   
  \label{fig:regerror-nn}             
\end{figure}

\section{Conclusions}
In this paper, we investigate model consistency and local linear convergence of iterative regularization methods, focusing on dual gradient descent (DGD) and its accelerated variant. Building on recent advances in partial smoothness theory, we prove that, in the small-noise regime, the output of iterative regularization can recover the underlying model of the ground truth. Moreover, we show that model consistency persists over an interval of stopping times, during which the residual error of the DGD scheme converges linearly. Numerical experiments are presented to corroborate our theoretical results.

\vskip2mm
\noindent\textbf{Acknowledgements}
JL is supported by the National Natural Science Foundation of China (No. 12201405), the ``Fundamental Research Funds for the Central Universities'', the National Science Foundation of China (BC4190065), Shanghai Municipal Science and Technology Key Project (No. 22JC1401500) and the Shanghai Municipal Science and Technology Major Project (2021SHZDZX0102).  
S. V. Acknowledges the support of the European Commission (grant TraDE-OPT 861137), of the European Research Council (grant SLING 819789 and grant MALIN 101117133) the US Air Force Office of Scientific Research (FA8655-22-1-7034), the Ministry of Education, University and Research (PRIN 202244A7YL project ``Gradient Flows and Non-Smooth Geometric Structures with Applications to Optimization and Machine Learning''), and the project PNRR FAIR PE0000013 - SPOKE 10. The research by S. V. has been supported by the MIUR Excellence Department Project awarded to Dipartimento di Matematica, Università di Genova, CUP D33C23001110001. S. V. is a member of the Gruppo Nazionale per l’Analisi Matematica, la Probabilità e le loro Applicazioni (GNAMPA) of the Istituto Nazionale di Alta Matematica (INdAM). This work represents only the view of the authors. The European Commission and the other organizations are not responsible for any use that may be made of the information it contains.

\noindent 
\ifnum\arxivorIP=1
\section*{References}
\fi
\bibliographystyle{unsrt}
\bibliography{bibtex}

\end{document}